\documentclass[a4paper]{article}
\usepackage[all]{xy}
\usepackage[latin1]{inputenc}        
\usepackage{graphics,graphicx}
\usepackage{amsfonts,amssymb,amsmath,color,mathrsfs,amstext,bm}
\usepackage{amsbsy, amsopn, amscd, amsxtra, amsthm,authblk}
\usepackage{enumerate,algorithmic,algorithm,adjustbox}
\usepackage{upref}
\usepackage{geometry}
\usepackage[mathlines]{lineno}
\usepackage{float}
\usepackage{yhmath}
\usepackage{booktabs}
\usepackage{subcaption}
\usepackage{multirow}
\usepackage{makecell}
\usepackage[colorlinks,
            linkcolor=red,
            anchorcolor=red,
            citecolor=red
            ]{hyperref}

\numberwithin{equation}{section}
\numberwithin{figure}{section}
\numberwithin{table}{section}

\def\R{\mathbb{R}}
\def\P{\mathbb{P}}

\newcommand{\expVariableNoise}{1D Variable Noise}
\newcommand{\expRareEvent}{1D Rare Event}
\newcommand{\expConstantNoise}{2D Constant Noise}
\newcommand{\expInteractingParticle}{64D Interacting Particle}

\newtheorem{theorem}{Theorem}[section]
\newtheorem{lemma}{Lemma}[section]
\newtheorem{proposition}{Proposition}[section]
\newtheorem{remark}{Remark}[section]

\newtheorem{definition}{Definition}[section]
\newtheorem{assumption}{Assumption}[section]
\newenvironment{keywords}{\par\smallskip\noindent\textbf{Keywords.} }{\par}
\newenvironment{MSCcodes}{\par\smallskip\noindent\textbf{MSC 2020.} }{\par}

\makeatletter
\@mparswitchfalse
\makeatother

\title{Deep operator learning for efficient sampling from invariant measures of stochastic differential equations}

\author[a]{
Ling Guo \thanks{Email: lguo@shnu.edu.cn}
}
\author[b,c]{
Lei Li \thanks{Email: leili2010@sjtu.edu.cn}
}
\author[b]{
Jingtong Zhang \thanks{Email: 
sffred@sjtu.edu.cn}
}

\affil[a]{Department of Mathematics, Shanghai Normal University, Shanghai, China}
\affil[b]{School of Mathematical Sciences,  Shanghai Jiao Tong University, Shanghai, China
}
\affil[c]{Institute of Natural Sciences, MOE-LSC, Shanghai Jiao Tong University, Shanghai, China}

\date{}

\begin{document}

\maketitle

\begin{abstract}
We introduce an amortized neural sampler that combines operator learning with flow methods for sampling. It maps SDE coefficient functions to pushforwards from a reference measure to the invariant measures, enabling efficient sampling across families of stochastic differential equations. Our framework shifts traditional sampling cost to an initial training phase, after which new SDE instances require only one encoder pass and a few ODE solver steps, independent of mixing time. To handle problems in high dimensions, we use Lagrangian trajectory sensors for the coefficient functions and cross attention in the architecture. We also theoretically establish the expressivity and resolution invariance of our framework. Experiments on 1D and 2D SDE families show competitive accuracy with substantial speedups over MCMC in regimes with slow mixing, transfer across sensor counts, and demonstration results on a 64D interacting particle SDE where traditional grid approaches are infeasible. 
\end{abstract}



\section{Introduction}\label{sec:intro}

Sampling from the invariant measures of Stochastic Differential Equations (SDEs) is a fundamental task in applied mathematics, statistical physics, statistics and computer science, which enables the computation of stationary expectations, free energies and rare event probabilities  for uncertainty quantification and stochastic control in related applications \cite{mattingly_ergodicity_2002,pavliotis_stochastic_2014,leimkuhler_matthews_2015}.
Traditional methods, such as Markov Chain Monte Carlo (MCMC) \cite{metropolis_equation_1953,hastings_monte_1970,robert_monte_2004}, rely on sequential simulations of the dynamics.   These methods are intrinsically sequential and generally produce correlated samples. They require an initial equilibration period and a sufficiently long simulation to approach equilibrium, while time discretization may introduce an additional bias. More importantly, for multimodal or metastable systems, transitions between separated regions can result in a prohibitively large mixing time and a small effective sample size. When the drift or diffusion coefficients vary across many problem instances, a new long simulation is usually required for every instance \cite{robert_monte_2004}, and this repeated mixing cost can become the dominant computational bottleneck.

Operator learning \cite{lu_deeponet_2021,li_fourier_2020,kovachki_neural_2023}, which directly learns the map between functions, is gaining interest in scientific machine learning, and   provides a natural framework for amortizing this cost over a family of stochastic systems. Consider the coefficient-indexed Stochastic Differential Equation (SDE) taking values in $\R^d$
\begin{equation}\label{eq:sde}
dX = b(X)\,dt+\sigma(X)\,dW,
\end{equation}
with $X_0\sim\mu_0$, drift $b:\mathbb{R}^d\to\mathbb{R}^d$, diffusion coefficient $\sigma:\mathbb{R}^d\to\mathbb{R}^{d\times d_W}$, and standard Brownian motion $W$ in $\mathbb{R}^{d_W}$. Both dimensions are fixed within each coefficient family.
Write $a:=(b,\sigma)$ for the pair of SDE coefficients throughout this paper, and denote its invariant measure by \(\pi_{a}\). Rather than constructing a separate sampler for each  \(a\), one may seek a generative sampler operator \(\mathcal{S}\) satisfying $\mathcal{S}(a,\cdot)_{\#}\gamma \approx \pi_{a}$,
where \(\gamma\) is a simple reference distribution and \(\#\) denotes the pushforward of measures. Thus, \(\mathcal{S}\) takes the drift and diffusion functions, together with a latent random variable, as inputs and returns a sample from the corresponding invariant law.  This provides a natural operator learning problem because the inputs \(a\) are functions. 
After an offline training stage over a suitable family of coefficients, such an operator could be evaluated on previously unseen SDEs and generate samples in parallel without repeating a long mixing simulation for every new query.

The generative sampler has been extensively studied in recent years, including normalizing flows and the continuous versions \cite{rezende_variational_2015,chen_neural_2018}; Boltzmann generators \cite{noe_boltzmann_2019}; score-based diffusion models \cite{song_score-based_2021}. The flow matching approach \cite{lipman_flow_matching_2023,lipman_flow_2024} provides a regression framework for training continuous normalizing flows by learning a target velocity field, while avoiding adversarial objectives. FlowKac uses temporal normalizing flows and the Feynman--Kac formula for Fokker--Planck equations \cite{elbekri_flowkac_2025}; the weak adversarial networks enforce the weak equation through a minimax problem involving trial and test networks \cite{zang_weak_adversarial_2020}, whereas the weak generative sampler uses randomized test functions to train a transport map directly and avoids the adversarial minimax optimization \cite{cai_wgs_sde_2024}. 
Although these methods can substantially reduce the cost of sampling a prescribed target, they are typically trained for a fixed distribution or a single stochastic system.  
The objective here is instead to combine generative sampling with operator learning.
Let us have also brief review of the neural operators. Representative neural operators include DeepONet \cite{lu_deeponet_2021}, which encodes input functions through branch networks and evaluates outputs through a trunk network; the Fourier Neural Operator \cite{li_fourier_2020,kovachki_neural_2023}, which parameterizes integral operator layers in the Fourier domain; multi-input operators (MIONet) \cite{jin_mionet_2022}, which accommodates several functional inputs through multiple branch networks; and Variable-Input Deep Operator Networks (VIDON) \cite{prasthofer_vidon_2022}, which allows the number and locations of input sensors to vary between samples.  Recently, generative neural operators have also been extended to learning probability laws on function spaces, first through adversarial neural operators \cite{rahman_gano_2022} and more recently through operator flow matching for stochastic process priors \cite{shi_ofm_2025,huang_lai_mfg_operator_2025}. Huang and Lai use an empirical measure limit for sample-based attention \cite{huang_lai_mfg_operator_2025}. Continuum attention also interprets discrete attention as an approximation of a population operator \cite{calvello_continuum_attention_2024}.

We propose an \textbf{amortized neural sampler} that uses neural operators and the flow matching \cite{lipman_flow_matching_2023}  to construct a reusable map from previously unseen drift and diffusion functions of the SDE \eqref{eq:sde} to samples from their corresponding invariant measures. 
Let $\mathcal{A}$ denote the space of admissible coefficient pairs. We assume $a$ is drawn from a distribution $\pi$ over $\mathcal{A}$, and we have access to existing samples $\{a_i, x_i\}$ generated via MCMC or other experimental methods. We want to learn a neural operator $v_\theta$ that functions as a conditional velocity field:
\begin{equation}\label{eq:flowode}
\frac{dx}{dt}=v_\theta[a](x, t),
\end{equation}
and apply it to unseen $a=(b, \sigma)$ to generate samples from the steady state distribution of the SDE defined by $a$.  Since the inference procedure requires evaluating $v_\theta$ multiple times to solve the ODE~\eqref{eq:flowode}, we adopt an encoder-decoder design: for a given pair $a=(b, \sigma)$, {\it the input functions are encoded once into a latent representation, which is then reused across all ODE integration steps}, thereby reducing the computational cost at inference time. We use the Perceiver architecture \cite{jaegle_perceiver_2021} to aggregate the embeddings of random Lagrangian trajectory sensors into a context vector of fixed dimension.

The main contributions of this work are as follows:
\begin{itemize}
    \item We formulate an operator-conditioned amortized sampler for parametric families of SDEs, which combines neural operator encoders with continuous normalizing flows. The high cost of training is incurred once, and the trained model then generates approximate invariant measure samples for unseen coefficient instances at an inference cost decoupled from the mixing time of the system.

    \item An architecture with attention uses random Lagrangian trajectory sensors for problems in high dimensions. The architecture accepts variable sensor counts during training and inference. The random trajectory sensors do not require a mesh and support flexible domains.

    \item We prove theoretically the expressivity of the operator samplers for compact sets of admissible coefficient functions, and also the resolution invariance properties for the proposed random Lagrangian attention sampler. We first establish that the continuity of the map from the coefficient functions to the flow that transforms a reference measure to the target measure.  The resolution invariance of random Lagrangian sensors with attention aggregation is established by uniform law of large numbers for compact attention classes.
\end{itemize}

The remainder of this paper is organized as follows. Section~\ref{sec:setup} introduces the mathematical setting and section \ref{sec:framework} presents the general framework. Section~\ref{sec:perceiver} describes the Lagrangian attention sampler in detail. Section~\ref{sec:theoryexpress} provides the theoretical analysis for the expressivity and section \ref{sec:resoinvar} establishes the resolution invariance. 
Lastly, Section~\ref{sec:experiments} reports the experimental evaluation.

\section{Problem Setting}\label{sec:setup}

We consider the problem of sampling from the invariant measures of a parametric family of SDEs of the form \eqref{eq:sde}, where the parameter is a set of functions (e.g., a coefficient field, a boundary condition, or a potential). In this work, the parameter is taken to be the coefficient functions $a=(b,\sigma)$.
The law $\mu_t$ of $X_t$ satisfies the following Fokker--Planck equation
\begin{gather}\label{eq:fp}
\partial_t\mu_t=-\nabla\cdot(b(x)\mu_t)+\frac{1}{2}\nabla^2:(\sigma(x)\sigma^\top(x)\mu_t).
\end{gather}
An invariant measure of the SDE is a stationary solution of the Fokker--Planck equation \eqref{eq:fp}. So the problem can be formulated as learning the map from the parameter $a=(b,\sigma)$ to the stationary solution of the Fokker-Planck equation \eqref{eq:fp}. Under standard dissipativity and regularity conditions on $(b,\sigma)$ --- made precise in Section~\ref{sec:theoryexpress} --- the SDE admits a unique invariant probability measure $\mu_a \in \mathcal{P}_2(\mathbb{R}^d)$, and the map $a \mapsto \mu_a$ is continuous \cite{bogachev2022fokker}. 

Since the object to be approximated is a probability distribution, we need a metric on distributions, and we use the quadratic Wasserstein distance throughout. On the space $\mathcal{P}_2(\mathbb{R}^d)$ of probability measures with finite second moment, it is defined by
\[
W_2(\mu, \nu) = \left(\inf_{\pi \in \Pi(\mu, \nu)} \int_{\mathbb{R}^d \times \mathbb{R}^d} |x - y|^2\, d\pi(x, y)\right)^{1/2},
\]
where $\Pi(\mu, \nu)$ is the set of couplings of $\mu$ and $\nu$, i.e.\ joint distributions on $\R^d\times \R^d$ with marginals $\mu$ and $\nu$. The pair $(\mathcal{P}_2(\mathbb{R}^d), W_2)$ is a complete metric space, and convergence in $W_2$ is equivalent to weak convergence together with convergence of second moments \cite{villani2008optimal}. The metric plays a double role in what follows: the approximation guarantees of Section~\ref{sec:theoryexpress} and Section~\ref{sec:resoinvar} are stated in $W_2$, and the experiments of Section~\ref{sec:experiments} use it as the quantitative quality criterion for the trained sampler.

It remains to state precisely what ``learning a sampler'' means for a family of SDEs. Fix a probability distribution $\pi$ on $\mathcal{A}$ representing the family of interest, e.g.\ a class of drift fields drawn from a prescribed prior. In its classical, one instance form, the sampling problem takes one fixed coefficient pair $a=(b,\sigma)$ and asks for i.i.d.\ samples $X^{(1)},\dots,X^{(N)}$ whose empirical law approximates $\mu_a$. MCMC and sequential Monte Carlo solve the problem in this form, one $a$ at a time, so each new coefficient pair triggers an independent --- and, in slow mixing regimes, long --- simulation of the dynamics. We consider instead the \emph{amortized} form of the problem: construct a single map
\[
\mathcal{S} : \mathcal{A} \times \mathbb{R}^d \to \mathbb{R}^d,\qquad (a,z) \mapsto \mathcal{S}(a,z),
\]
where $z$ is a latent variable drawn from a fixed reference distribution $\gamma$ on $\mathbb{R}^d$ (typically $\gamma = \mathcal{N}(0,I_d)$), such that the pushforward $\hat\mu_a := \mathcal{S}(a,\cdot)_\#\gamma$ approximates $\mu_a$ for every $a$ in the support of $\pi$. The quality of $\mathcal{S}$ is measured by the population error
\begin{equation}\label{eq:amortized_objective}
\mathcal{E}(\mathcal{S}) \;=\; \mathbb{E}_{a \sim \pi}\bigl[W_2(\hat\mu_a, \mu_a)\bigr].
\end{equation}
The amortized formulation pays an upfront construction cost in exchange for cheap evaluation: once $\mathcal{S}$ is built, sampling from $\mu_a$ for a previously unseen $a$ requires only forward evaluations of $\mathcal{S}$, with no per-instance simulation of the underlying SDE. 

\section{The general framework for neural operator sampler}\label{sec:framework}


We now explain our parametric sampler framework promised by Section~\ref{sec:setup}. We formalize the \textbf{neural operator sampler} as follows.

\begin{definition}[Neural Operator Sampler]
Let $\mathbb{R}^d$ be a state space, and $\mathcal{A}$ be a Banach space of input functions. A Neural Operator Sampler is a parameterized mapping
\begin{equation}
\mathcal{S}_\theta: \mathcal{A} \times \mathbb{R}^d \to \mathbb{R}^d,
\end{equation}
where $\theta\in\Theta$ denotes the learnable parameters (weights). For a given input $a=(b,\sigma) \in\mathcal{A}$, the map $\mathcal{S}_\theta(a,\cdot)$ pushes forward the reference measure $\gamma$ (e.g., standard Gaussian) to a target distribution $\hat\mu_a\in\mathcal{P}(\mathbb{R}^d)$.
\end{definition}

To construct the map $\mathcal{S}_\theta$, we adopt the Continuous Normalizing Flow (CNF) paradigm and learn the transport dynamics from reference samples to target samples. We introduce a Neural Operator $\mathcal{N}_\theta$ that approximates the velocity field governing the transport from the reference measure $\gamma$ to $\mu_a$,
\begin{gather}
\mathcal{N}_\theta: \mathcal{A} \to C([0,1] \times \mathbb{R}^d; \mathbb{R}^d).
\end{gather}
For a specific input $a$, the operator predicts a velocity field
\begin{gather}
v_a(x,t) = \mathcal{N}_\theta(a)(x,t).
\end{gather}
The sampler $\mathcal{S}_\theta(a, z)$ is then defined as the solution $\phi(1)$ of the ODE
\begin{equation}\label{eq:cnf_ode}
\frac{d\phi}{dt} = v_a(\phi(t), t), \quad \phi(0) = z \sim \gamma.
\end{equation}
We use the alternative notation $v[b,\sigma](\cdot,\cdot)$ to indicate the functional dependence on $b$ and $\sigma$.

A natural question is whether such a CNF exists for each coefficient pair $a$, and whether some generating velocity field can serve as a regression target during training. We recall two classical existence mechanisms for such flows: the Benamou--Brenier optimal transport flow, which is exact but depends on the optimal coupling, unavailable from data; and the marginal flow matching construction, which builds a valid transport flow from any coupling of source and target samples and underlies our training objective.


\begin{enumerate}[(1)]
\item Benamou--Brenier optimal transport flow.

The dynamic (Benamou--Brenier) formulation of optimal transport~\cite{benamou_computational_2000} looks for a den\-sity--velocity pair $(\rho_t,v_t)$ that depends on time and solves the constrained optimization problem
\begin{equation}\label{eq:bb}
W_2^2(\gamma,\mu_a) = \inf_{(\rho,v)}\left\{
\begin{aligned}
&\int_0^1\!\int_{\mathbb{R}^d}\|v(x,t)\|^2\,\rho(x,t)\,dx\,dt \;:\; \\
&\partial_t\rho+\nabla\cdot(\rho\, v)=0,\;\rho_0=\gamma,\;\rho_1=\mu_a
\end{aligned}
\right\}.
\end{equation}
The optimal velocity field $v^{\mathrm{BB}}$ transports $\gamma$ to $\mu_a$ along the $W_2$-geodesic with minimal kinetic energy. By Brenier's theorem~\cite{brenier1991polar}, when $\gamma$ is absolutely continuous there exists a unique $W_2$-optimal transport map $T=\nabla\psi$ (with $\psi$ convex) satisfying $T_\#\gamma=\mu_a$. The Benamou--Brenier geodesic is then realized by the linear interpolation
\[
X_t = (1-t)X_0 + t\,T(X_0), \qquad X_0\sim\gamma,
\]
with marginal velocity field
\[
v^{\mathrm{BB}}(x,t) = \mathbb{E}\!\left[T(X_0)-X_0 \mid X_t=x\right].
\]
The optimal transport construction thus provides an exact, geometrically natural flow from $\gamma$ to $\mu_a$. Its drawback is practical: it requires access to the optimal map or coupling for each target instance, which is generally not available during training --- our data consist of \emph{unpaired} samples from $\gamma$ and $\mu_a$.

\item Marginal flow matching construction.

A second existence mechanism is the one underlying flow matching, and it removes this requirement: it builds a valid transport flow from \emph{any} coupling $\eta_a\in\Pi(\gamma,\mu_a)$ of source and target samples --- in particular the independent coupling $\gamma\otimes\mu_a$, which requires only unpaired samples. The idea is to prescribe simple dynamics conditionally on a pair of endpoints and then average. For each pair $(x_0,x_1)$, choose a conditional interpolation path $\rho_t(\cdot\mid x_0,x_1)$ and a conditional velocity $u_t(\cdot\mid x_0,x_1)$ satisfying the conditional continuity equation
\[
\partial_t \rho_t(x\mid x_0,x_1)
+\nabla\cdot\!\left(\rho_t(x\mid x_0,x_1)u_t(x\mid x_0,x_1)\right)=0,
\]
with endpoints concentrated at $x_0$ and $x_1$. Averaging over the coupling gives a marginal path and a marginal velocity:
\[
\rho_t^a(x)=\int \rho_t(x\mid x_0,x_1)\,d\eta_a(x_0,x_1),
\]
\[
\rho_t^a(x)u_t^a(x)=\int \rho_t(x\mid x_0,x_1)u_t(x\mid x_0,x_1)\,d\eta_a(x_0,x_1),
\]
where $u_t^a$ is defined arbitrarily on the set where $\rho_t^a=0$. Integrating the conditional continuity equations yields
\[
\partial_t \rho_t^a+\nabla\cdot(\rho_t^a u_t^a)=0,\qquad \rho_0^a=\gamma,\quad \rho_1^a=\mu_a.
\]
Equivalently, for $X_t\sim \rho_t^a$,
\[
u_t^a(x)=\mathbb{E}\!\left[u_t(X_t\mid X_0,X_1)\mid X_t=x\right],
\]
which is the pointwise $L^2$ projection of the conditional velocity onto functions of the current state. This field has a dynamical role: it generates the marginal path and defines a valid transport flow from $\gamma$ to $\mu_a$. The conditional expectation form is also what makes the construction trainable: it characterizes $u_t^a$ as the minimizer of a regression problem whose samples require only draws from the coupling and the conditional path, which turns into the training objective (see the details below).
\end{enumerate}


Flow matching trains the velocity field by pointwise regression and avoids ODE solves inside each gradient step, giving it a lower online training cost than maximum likelihood CNF training. The price is that it requires training samples from the target invariant measures; in this work we focus on this supervised setting. Stochastic interpolants~\cite{albergo_stochastic_2023} provide a closely related formulation, and the expressivity theory of Section~\ref{sec:theoryexpress} uses a smoothed variant of the flow matching construction.

For a fixed coefficient pair $a$, the ideal marginal flow matching loss is
\begin{equation}\label{eq:fm_objective}
\mathcal{L}_{\mathrm{FM}}^a(\theta)
= \mathbb{E}_{t\sim\mathcal{U}[0,1]}\,\mathbb{E}_{x\sim \rho_t^a}
\bigl\|v_\theta[a](x,t)-u_t^a(x)\bigr\|^2.
\end{equation}
The marginal path $\rho_t^a$ and velocity $u_t^a$ are generally intractable. Conditional flow matching~\cite{lipman_flow_2024} replaces~\eqref{eq:fm_objective} with a tractable conditional regression problem. The common $x_1$-conditioned Gaussian path corresponds to the independent source--target coupling: for each target sample $x_1\sim\mu_a$, one defines a conditional probability path $\rho_t(\cdot\mid x_1)$ and a conditional velocity field $u_t(\cdot\mid x_1)$ satisfying
\[
\partial_t \rho_t(x\mid x_1)+\nabla\cdot\!\bigl(\rho_t(x\mid x_1)\,u_t(x\mid x_1)\bigr)=0, \qquad \rho_0(\cdot\mid x_1)=\gamma,\quad \rho_1(\cdot\mid x_1)\approx\delta_{x_1}.
\]
The resulting conditional flow matching objective
\begin{equation}\label{eq:cfm_objective}
\mathcal{L}_{\mathrm{CFM}}(\theta)
=\mathbb{E}_{a\sim\pi}\mathbb{E}_{t\sim\mathcal{U}[0,1]}\mathbb{E}_{x_1\sim\mu_a}
\mathbb{E}_{x\sim \rho_t(\cdot\mid x_1)}
\bigl\|v_\theta[a](x,t)-u_t(x\mid x_1)\bigr\|^2
\end{equation}
has the same gradients as the corresponding marginal loss for this
independent coupling path~\eqref{eq:fm_objective} \cite{lipman_flow_2024} and
requires only samples from the simple conditional path.

For the independent coupling, this paper uses the standard Gaussian linear conditional path
\begin{equation}\label{eq:ot_cond_path}
\rho_t(x\mid x_1)=\mathcal{N}\!\bigl(x\;\big|\;t\,x_1,\,(1-t)^2 I\bigr), \qquad u_t(x\mid x_1)=\frac{x_1-x}{1-t}.
\end{equation}
Drawing $X_0\sim\gamma=\mathcal{N}(0,I)$ independently of $X_1$ and setting $X_t=(1-t)X_0+tX_1$ gives the equivalent velocity target $u_t(X_t\mid X_1)=X_1-X_0$.

More generally, if a dependent source--target coupling $\eta_a\in\Pi(\gamma,\mu_a)$ is used to form training pairs, the conditional path should be viewed as conditioned on the pair $(X_0,X_1)$:
\[
    X_t=(1-t)X_0+tX_1,\qquad u_t(X_t\mid X_0,X_1)=X_1-X_0 .
\]
The corresponding supervised training loss is
\begin{equation}\label{eq:linear_cfm_objective}
\mathcal{L}_{\mathrm{CFM}}^{\mathrm{lin}}(\theta)
= \mathbb{E}_{a\sim\pi}\mathbb{E}_{(X_0,X_1)\sim\eta_a}\mathbb{E}_{t\sim\mathcal{U}[0,1]}
\Bigl[\bigl\|v_\theta[a](X_t,t)-(X_1-X_0)\bigr\|^2\Bigr].
\end{equation}
The default independent pairing corresponds to $\eta_a=\gamma\otimes\mu_a$; one may use an empirical or minibatch optimal transport coupling when one wants a closer approximation to the Benamou--Brenier geometry. With an optimal source--target coupling, the linear construction conditioned on pairs realizes the Benamou--Brenier geodesic~\eqref{eq:bb}; with independent source and target samples, it is a tractable flow matching interpolant. In either case, different couplings and interpolation paths lead to different valid training targets.



In order to mitigate the cost of inference the trained model multiple times to solve the flow ODE~\eqref{eq:cnf_ode}, we propose an encoder--decoder architecture to construct the neural operator $v_\theta$. The encoder maps the input functions $(b, \sigma)$ to a latent representation, and the decoder takes the latent code along with spatiotemporal coordinates $(x,t)$ to produce the velocity field. Different encoders and decoders can be used to fit different problems. Specifically, the calculation of $v_\theta[b, \sigma](x,t)$ may be decomposed into two steps:
\begin{align}
z & = \text{Encoder}(b, \sigma; \theta_e),\label{eq:encoder} \\
v_\theta[b, \sigma](x,t) & = \text{Decoder}(z, x, t; \theta_d),
\end{align}
where $\theta = (\theta_e, \theta_d)$ are the learnable parameters of encoder and decoder respectively.
In the subsequent sections, we will discuss several possible representations of the input functions $(b,\sigma)$ and the corresponding encoder and decoders.

Training uses the conditional flow matching loss in~\eqref{eq:linear_cfm_objective}. At inference, we evaluate the model multiple times to solve the ODE \eqref{eq:flowode}. The latent representation \eqref{eq:encoder} is computed once per coefficient pair $a=(b, \sigma)$, which reduces time and memory cost.
The complete training and inference procedures are summarized in Algorithm~\ref{alg:training_inference}.

\begin{algorithm}[htbp]
\caption{Amortized Neural Sampler using flow matching loss: Training and Inference}\label{alg:training_inference}
\begin{algorithmic}[1]
\STATE \textbf{Input:} Training set $\mathcal{D}=\{(a_i, \{x_i^{(j)}\}_{j=1}^{N_s})\}_{i=1}^{N_f}$ with $a_i=(b_i,\sigma_i)$; reference measure $\gamma$; ODE solver with $K$ steps.
\STATE \textbf{Output:} Trained parameters $\theta=(\theta_e,\theta_d)$.
\STATE \textit{\% Training}
\FOR{each minibatch}
  \STATE Sample coefficients $a\sim\pi$, target $X_1\sim\mu_a$, noise $X_0\sim\gamma$, time $t\sim U[0,1]$.
  \STATE Compute interpolant $X_t = (1-t)X_0 + tX_1$.
  \STATE Encode: $z \gets \mathrm{Encoder}(a;\,\theta_e)$.
  \STATE Predict velocity: $\hat{v} \gets \mathrm{Decoder}(z, X_t, t;\,\theta_d)$.
  \STATE Update $\theta$ by $\nabla_\theta \|\hat{v} - (X_1 - X_0)\|^2$.
\ENDFOR

\STATE \textit{\% Inference (for a new, unseen $a^*$)}
\STATE Encode once: $z^* \gets \mathrm{Encoder}(a^*;\,\theta_e)$.
\STATE Draw $X_0 \sim \gamma$.
\FOR{$k = 1,\dots,K$}  
  \STATE ODE integration (e.g.\ Euler, RK45): $X_{t_{k}} \gets X_{t_{k-1}} + \Delta t_k\;\mathrm{Decoder}(z^*, X_{t_{k-1}}, t_{k-1};\,\theta_d)$.
\ENDFOR
\STATE Return $X_1$ as a sample from the learned invariant measure $\hat\mu_{a^*}$.
\end{algorithmic}
\end{algorithm}



\section{A Lagrangian attention operator sampler}\label{sec:perceiver}

The classical grid approach implements the encoder--decoder framework with a Multi-input DeepONet~\cite{jin_mionet_2022}.
At $m$ grid sensors, we record $b_j=b(x_j)\in\mathbb{R}^d$ and $\sigma_j=\sigma(x_j)\in\mathbb{R}^{d\times d_W}$. We retain $p\le d$ rows of $\sigma$, omitting only rows that vanish for every coefficient instance and position. Stacking each retained row across sensors gives the input tensors:
\[
\hat b \in \mathbb{R}^{m\times d}, \quad \hat \sigma_k \in \mathbb{R}^{m\times d_W},\quad k=1,\dots,p.
\]

The velocity field is assembled via the tensor product structure of the Multi-input DeepONet \cite{jin_mionet_2022}: 
\begin{equation}
v_\theta[b, \sigma](x,t) = \mathcal{G}(\hat b,\hat\sigma_k,x, t; \theta) = \sum_{l=1}^q f^0_l(\hat b)\odot f^1_l(\hat\sigma_1)\odot\cdots\odot f^{p}_l(\hat\sigma_p)\odot u_l(\hat x,\hat t),
\end{equation}
where $\odot$ denotes the elementwise (Hadamard) product, and \(\hat{x},\hat{t}\) are \(x, t\) embedded via a fixed sinusoidal map.  In the notation of Section~\ref{sec:framework}, the Branch nets collectively form the encoder 
\[z = (f^0(\hat b), f^1(\hat\sigma_1), \ldots, f^p(\hat\sigma_p)),\]
and the Trunk net together with the product-sum formula form the decoder.

For problems in high dimensions a grid of sensors is infeasible: the number of grid points grows exponentially in the dimension $d$, so the dense Eulerian sampling of $b(x),\sigma(x)$ used by classical neural operators~\cite{kovachki_neural_2023,li_fourier_2020} exceeds a realistic budget. Motivated by VIDON for random point sensors \cite{prasthofer_vidon_2022}, we aggregate information from \textbf{random Lagrangian sensors}---mobile probe trajectories that follow a cheap proxy of the dynamics---and encode them with attention. See also \cite{feng_learn_evolve_2026} for related ideas.
Because the encoder reads the sensors through their \emph{empirical measure}, it ingests a variable number of sensors $m$ and runs at a sensor count $m'\neq m$ without retraining. As $m$ grows, the output law of the model with fixed weights converges---in expectation over the random probes---to its own population-probe limit (Section~\ref{sec:resoinvar}); increasing $m$ reduces the Monte Carlo error from the finite probe set relative to that limit.

The full architecture is shown in Figure~\ref{fig:arch_attention}. The model embeds $m$ sensor trajectories into the token list $E$; $E$ then interacts repeatedly with a learnable latent token list $Z_0$ of fixed size to obtain the latent context $Z_L$. $Z_L$, together with $(X, t)$ after passing through a Gaussian Fourier Projection, goes through an AdaLN Decoder and produces the final velocity field $v_\theta[a](X, t)$. The following subsections describe the exact operations.

\begin{figure}[htbp]
\centering
\includegraphics[width=0.8\textwidth]{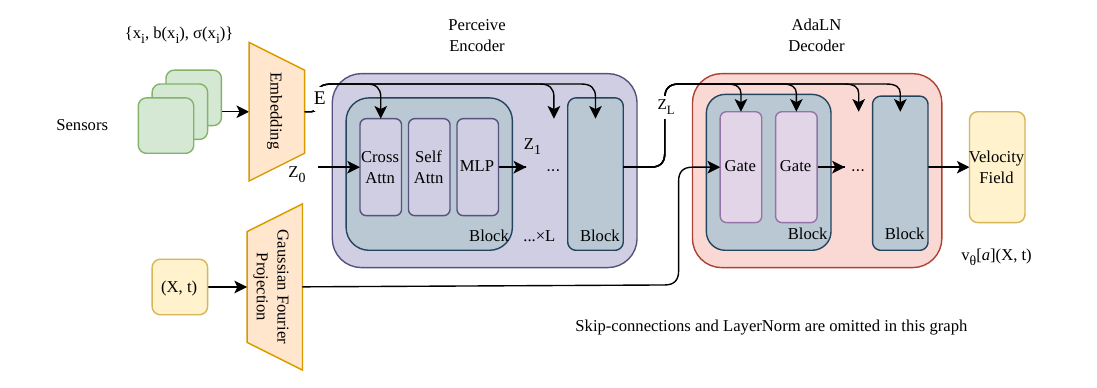}
\caption{Architecture of the amortized neural sampler with an attention encoder.\label{fig:arch_attention}}
\end{figure}

\subsection{Lagrangian trajectory sensors and the per-sensor embedding}

We draw $m$ random initial positions $\{x_i\}_{i=1}^{m}$ i.i.d.\ from a fixed probability distribution on $\mathbb{R}^d$ (the probe initial law); the count $m$ may differ across data samples. From each $x_i$ we simulate a short probe trajectory $x_{i,0},\dots,x_{i,L}$ and record at every step the local feature
\[
y_{i,j}=\bigl(x_{i,j},\;b(x_{i,j}),\;\operatorname{vec}\sigma(x_{i,j})\bigr)\in\mathbb{R}^{\,2d+dd_W},
\qquad j=0,\dots,L,
\]
which pairs each visited state with the drift and diffusion sampled there. The $i$-th \emph{probe observation} is the ordered collection
\[
Y_i=(y_{i,0},\dots,y_{i,L})\in\mathbb{R}^{(L+1)\times c_0},\qquad c_0=2d+dd_W.
\]
Known structure in $\sigma$ reduces the feature width. For $d_W=d$ and isotropic diffusion $\sigma=\sigma_0 I_d$, we use $y_{i,j}=(x_{i,j},\,b(x_{i,j}),\,\sigma_0)\in\mathbb{R}^{\,2d+1}$.

Sensing with trajectories enriches pointwise evaluation by recording local changes of the drift and diffusion along a short path; even short trajectories provide finite difference information about the local vector field, and their ordered structure can be used by a convolutional encoder.

The probe dynamics need not coincide with the target dynamics: any process that visits the domain and permits pointwise evaluation of $b$ and $\sigma$ can drive the probes. In practice we drive them either by a cheap exploratory process---Brownian motion with a fixed step size, or a constant diffusion proxy $dX = b(X)\,dt + \bar\sigma\,dW$---or by the original SDE \eqref{eq:sde} itself, a natural choice when the drift steers probes toward dynamically relevant regions. In either case the sensing cost is decoupled from the difficulty of the target SDE, and the probes can be extremely short. In several of our experiments (Section~\ref{sec:experiments}) trajectories of only $L=20$ steps---and in some settings as few as $L=1$ (two points)---already suffice for accurate reconstruction of the invariant measure. Because each probe involves only $L$ evaluations of $b$ and $\sigma$, the total sensing cost is negligible compared to the $\mathcal{O}(\tau_{\mathrm{mix}}/\Delta t)$ steps required by MCMC.

After the Lagrangian probes obtained, we further make use of a small \textbf{1D convolutional neural network} acting along the step axis to map each observation into a vector of fixed width. Writing $C^{(0)}=Y_i$, the network applies $N$ convolutional blocks
\[
C^{(\ell)}=\mathrm{GELU}\!\Bigl(\mathrm{LayerNorm}\bigl(\mathrm{Conv1d}^{(\ell)}_{k=3}(C^{(\ell-1)})\bigr)\Bigr)\in\mathbb{R}^{(L+1)\times c_\ell},
\qquad \ell=1,\dots,N,
\]
followed by global average pooling over the $L+1$ steps,
\[
e_i=\frac{1}{L+1}\sum_{j=0}^{L}C^{(N)}_{j,:}\in\mathbb{R}^{d_e},\qquad d_e:=c_N.
\]
Here GELU donates the Gaussian Error Linear Unit activation
\[
\mathrm{GELU}(x)=x\Phi(x), \quad \Phi(x)=\frac{1}{\sqrt{2\pi}}\int_{-\infty}^x \exp\left(-\frac{|y|^2}{2}\right)dy.
\]
The convolutions with kernel size $3$ use ``same'' padding, which preserves the step length $L+1$ through all $N$ blocks and keeps the embedding defined down to $L=1$; they capture local temporal patterns such as the curvature and magnitude of the drift, and the pooling produces an output $e_i$ of fixed size for any trajectory length $L$, so the embedding dimension $d_e$ is independent of $L$ and $m$. (A small MLP on a pooled trajectory is an alternative; the 1D CNN consistently outperformed it in our experiments by exploiting the sequential structure.)

The encoder downstream sees the sensors \emph{only} through their empirical measure
\[
\hat\nu_m=\frac1m\sum_{i=1}^{m}\delta_{e_i},\qquad e_i\in\mathbb{R}^{d_e}.
\]
This is the object the resolution invariance theory acts on: increasing $m$ sharpens $\hat\nu_m$ toward an underlying sensor law $\nu$, and the encoder applies the \emph{same} map to $\hat\nu_m$ and to $\nu$.

\subsection{Attention Perceiver encoder}\label{subsec:perceiver_encoder}

The encoder aggregates the sensor embeddings into a context vector of fixed dimension. This aggregation accepts variable sensor counts and is invariant to sensor permutations.

To perform this aggregation, the encoder carries a fixed set of $k$ \textbf{latent tokens} (typically $k\ll m$), stacked as $Z\in\mathbb{R}^{k\times d_z}$ of latent width $d_z$. Attention updates these tokens with sensor information, and the final readout converts them into the context vector. The encoder represents the $m$ sensors by projecting each embedding $e_i\in\mathbb{R}^{d_e}$ to width $d_z$ with a shared linear map $P_e\in\mathbb{R}^{d_e\times d_z}$, stacking the rows $E_i=e_iP_e$ into $E\in\mathbb{R}^{m\times d_z}$.

Two attention operations act on these arrays: cross attention, in which the latents read the sensors, and self attention, in which the latents exchange information among themselves. Each operation uses $H$ heads of width $d_h=d_z/H$.

Both operations are built from the same attention primitive. A single head maps a query array $Q$, key array $K$, and value array $V$ (rows of width $d_h$) to
\[
    \mathrm{Atten}(Q,K,V)=\mathrm{softmax}\!\Bigl(\frac{QK^\top}{\sqrt{d_h}}\Bigr)V,
    \qquad
    \Bigl[\mathrm{softmax}\!\Bigl(\tfrac{QK^\top}{\sqrt{d_h}}\Bigr)\Bigr]_{rj}
    =\frac{\exp(\langle q_r,k_j\rangle/\sqrt{d_h})}{\sum_{i}\exp(\langle q_r,k_i\rangle/\sqrt{d_h})},
\]
where $q_r,k_j$ are the rows of $Q,K$ and the sum runs over the rows of $K$. The softmax normalizes each row, so row $r$ of $\mathrm{Atten}(Q,K,V)$ is the convex combination of the rows of $V$ weighted by query $r$, and the factor $\sqrt{d_h}$ holds the scores at a stable scale as $d_h$ grows.

Each of the $H$ heads carries learnable projections $W^Q_h,W^K_h,W^V_h\in\mathbb{R}^{d_z\times d_h}$, and the $H$ head outputs are concatenated and mixed by $W^O\in\mathbb{R}^{d_z\times d_z}$. Cross attention draws queries from the latents, the keys and values from the sensors; self attention draws all three from the latents:
\[
\begin{aligned}
    \mathrm{CA}(Z,E)&=\bigl[\,\mathrm{Atten}(ZW^Q_h,\,EW^K_h,\,EW^V_h)\,\bigr]_{h=1}^{H}\,W^O,\\
    \mathrm{SA}(Z)&=\bigl[\,\mathrm{Atten}(ZW^Q_h,\,ZW^K_h,\,ZW^V_h)\,\bigr]_{h=1}^{H}\,W^O,
\end{aligned}
\]
both in $\mathbb{R}^{k\times d_z}$, where $[\,\cdot\,]_{h=1}^{H}$ concatenates the head outputs along the feature axis and cross- and self attention carry separate projections. Cross attention carries the information of the $m$ sensors into the $k$ latents; self attention lets the latents share information.


Starting from the learned initial latents $Z^{(0)}\in\mathbb{R}^{k\times d_z}$, each of the $L_{\mathrm{enc}}$ blocks reads the sensors into the latents through cross attention, lets the latents exchange information through self attention, and refines them tokenwise through an MLP with two layers of width $4d_z$ and GELU activation, each step carried by a residual connection and a post-LayerNorm:
\[
\begin{aligned}
Z' &= \mathrm{LN}\bigl(Z^{(\ell)}+\mathrm{CA}(Z^{(\ell)},E)\bigr), \\
Z'' &= \mathrm{LN}\bigl(Z'+\mathrm{SA}(Z')\bigr), \\
Z^{(\ell+1)} &= \mathrm{LN}\bigl(Z''+\mathrm{MLP}(Z'')\bigr),
\end{aligned}
\qquad \ell=0,\dots,L_{\mathrm{enc}}-1 .
\]
Every block reuses the same sensor array $E$ as keys and values and updates the latent tokens. Each cross attention head returns a ratio of two sums over the sensors---the value average in the numerator and the softmax normalization in the denominator---so its value is a function of the empirical measure $\hat\nu_m$ alone. Section~\ref{sec:resoinvar} and Appendix~\ref{app:randomresolution} build on this empirical measure structure, replacing $\hat\nu_m$ by a general sensor law $\nu$ to establish resolution invariance as $m$ grows.

After the $L_{\mathrm{enc}}$ blocks, the encoder reads out the $k$ latent tokens by flattening them and projecting linearly to the context vector
\[
z=\operatorname{Encoder}(\hat\nu_m):=W^{\mathrm{out}}\,\operatorname{vec}\bigl(Z^{(L_{\mathrm{enc}})}\bigr)\in\mathbb{R}^{d_z},
\qquad W^{\mathrm{out}}\in\mathbb{R}^{d_z\times kd_z}
\]
(for $k=1$ this is just the single token). The whole construction is permutation invariant in the ordering of the sensor embeddings~\cite{zaheer_deep_2017}. In our experiments we use $k=4$ latent tokens, $H=4$ heads of width $d_h=32$ (so $d_z=128$), and $L_{\mathrm{enc}}=2$ blocks, yielding $z\in\mathbb{R}^{128}$.

\subsection{AdaLN-conditioned flow decoder}

The decoder receives the context $z\in\mathbb{R}^{d_z}$, a state $x\in\mathbb{R}^d$, and the flow time $t\in[0,1]$, and outputs the velocity $v_\theta[b,\sigma](x,t)\in\mathbb{R}^d$.

The flow time $t$ enters through a fixed sinusoidal embedding $\tau=\mathrm{SinEmb}(t)\in\mathbb{R}^{d_t}$, whose $j$-th entry is
\[
\tau_{2j-1}=\sin(2\pi f_j\,t),\qquad \tau_{2j}=\cos(2\pi f_j\,t),\qquad f_j=f_{\min}\!\left(\tfrac{f_{\max}}{f_{\min}}\right)^{\!(j-1)/(d_t/2-1)},
\]
with geometrically spaced frequencies $f_1=f_{\min}$ to $f_{d_t/2}=f_{\max}$ and $j=1,\dots,d_t/2$. The context that modulates every block fuses the operator latent $z$ with this time embedding through a learned linear projection,
\[
c = W_c\,[z;\,\tau]+b_c\in\mathbb{R}^{d_z},
\]
and the state and time embedding are concatenated and lifted to the hidden width $D$,
\[
h^{(0)}=W_{\mathrm{in}}\,[x;\,\tau]+b_{\mathrm{in}}\in\mathbb{R}^{D}.
\]

We condition the decoder through \textbf{adaptive layer normalization} (AdaLN)~\cite{peebles_scalable_2023}, letting the context $c$ set the scale and shift of every block's normalization so that the conditioning signal reaches all depths. Each of the $N_d$ blocks modulates its two normalization layers with scale and shift predicted from $c$,
\[
\mathrm{AdaLN}(h;c)=\bigl(1+\gamma(c)\bigr)\odot\mathrm{LN}(h)+\beta(c),
\]
and applies two such modulations with a residual connection,
\[
\begin{aligned}
h' &= \mathrm{act}\!\bigl(\mathrm{AdaLN}(h^{(\ell)};\,c)\bigr), \\
h'' &= \mathrm{act}\!\bigl(\mathrm{AdaLN}(W_1 h';\,c)\bigr), \\
h^{(\ell+1)} &= h^{(\ell)} + \alpha(c)\odot W_2 h'',
\end{aligned}
\qquad \ell=0,\dots,N_d-1,
\]
where $\gamma,\beta,\alpha$ are linear functions of $c$ and $\mathrm{act}$ is Mish.

Finally, after $N_d$ blocks the hidden state is normalized and projected to the velocity,
\[
v_\theta[b,\sigma](x,t)=W_v\,\mathrm{LN}(h^{(N_d)})\in\mathbb{R}^d.
\]
The output weight $W_v$ is initialized to zero so that the velocity field starts near zero at the beginning of training. For interacting particle systems with state $X=(x_1,\ldots,x_N)\in(\mathbb{R}^{d_p})^N$, we use the particle variant described in Section~\ref{subsec:particles}: the decoder reshapes $X$ into $N$ particle tokens, embeds each token with shared weights, applies AdaLN-conditioned self attention and feedforward blocks across the particle tokens, and projects each token to its velocity $v_i$. Because the token maps are shared and self attention is permutation equivariant, relabeling the particles relabels the output velocities in the same way.


\section{Expressivity of the Lagrangian attention sampler}\label{sec:theoryexpress}

This section is devoted to the expressivity result.
The sampler is intended to learn a map from SDE coefficients to invariant measures.
The theoretical question is whether our architecture has sufficient expressive power to approximate this map uniformly over a compact family of coefficient functions.
The reduction theorem establishes this result under the stated assumptions via neural operator approximation of a deterministic coefficient-dependent velocity field on compact domains.
Appendix~\ref{app:proofs} contains the proofs and construction details.
The main text records the assumptions, the key lemmas and propositions, the
theorem statements, and the proof ideas needed to interpret them.

For \(a=(b,\sigma)\) with \(\sigma:\mathbb R^d\to\mathbb R^{d\times d_W}\), write
\[
    A_a(x)=\sigma(x)\sigma(x)^\top\in\mathbb R^{d\times d}
\]
for the diffusion matrix. The generator of the SDE acts on smooth test functions by
\begin{gather}
\mathcal L_a\varphi(x)
= b(x)\cdot\nabla\varphi(x) +\frac12\operatorname{Tr}\!\bigl(A_a(x)D^2\varphi(x)\bigr).
\end{gather}
The generator depends on \(\sigma\) only through \(A_a\). As a result, the invariant measure \(\mu_a\) also depends on \(\sigma\) only through \(A_a\). We keep the pair \(a=(b,\sigma)\) as the variable because the sampler receives \(b\) and \(\sigma\) as its input.

The coefficient space is
\begin{gather}
\mathcal X_a
    = C(\mathbb R^d;\mathbb R^d)
    \times C(\mathbb R^d;\mathbb R^{d\times d_W}).
\end{gather}
We equip this space with the compact-open topology, that is, uniform convergence
on compact subsets. A sequence \(a_n=(b_n,\sigma_n)\) converges to
\(a=(b,\sigma)\) if, on every compact set \(B\subset\mathbb R^d\),
\[
    \sup_{x\in B}|b_n(x)-b(x)|
    +
    \sup_{x\in B}\|\sigma_n(x)-\sigma(x)\|
    \longrightarrow 0 .
\]
Choose any compact exhaustion \((B_j)_{j\ge1}\) of \(\mathbb R^d\), and define
\[
    p_j(a,a')
    =
    \sup_{x\in B_j}|b(x)-b'(x)|
    +
    \sup_{x\in B_j}\|\sigma(x)-\sigma'(x)\|.
\]
The metric
\[
    d_{\mathcal X_a}(a,a')
    =
    \sum_{j=1}^\infty 2^{-j}\min\{1,p_j(a,a')\}
\]
metrizes the compact-open topology on \(\mathcal X_a\). If \(a_n\to a\) in this topology, then \(A_{a_n}\to A_a\) uniformly on every compact set, because \(\sigma_n\) is uniformly bounded on that set.

In this topology, compactness of a set \(E\subset\mathcal X_a\) means that
every sequence in \(E\) has a subsequence converging uniformly on every compact
subset of \(\mathbb R^d\). The Arzela--Ascoli criterion gives the following
standard compactness condition.

\begin{proposition}[Compactness criterion]\label{prop:compactness_criterion}
A closed set \(E\subset\mathcal X_a\) is compact if, for each ball
\(B_m=\{x:|x|\le m\}\), the restrictions of its elements to \(B_m\) are
uniformly bounded and equicontinuous.
\end{proposition}

\subsection{Invariant measure and the properties}

We will focus on the following set of coefficients with parameters $\nu, p_0, \lambda, M$.
Let \(\mathcal A(\nu,p_0,\lambda,M)\) be the class of coefficients
\(a=(b,\sigma)\in\mathcal X_a\) such that:
\begin{enumerate}
    \item \textbf{Uniform ellipticity} (which requires $d_W\ge d$).
    \[
        A_a(x)\succeq \nu I,
        \qquad x\in\mathbb R^d .
    \]
    \item \textbf{Khasminskii condition for polynomial Lyapunov functions.}
    There exists \(p_0>2\) such that
    \[
        \langle b(x),x\rangle
        +\frac12(p_0-1)\operatorname{Tr}A_a(x)
        \le -\lambda |x|^2+M,
        \qquad x\in\mathbb R^d .
    \]
    In particular, for \(V(x)=|x|^p\), \(p\le p_0\), this condition yields
    \[
        \mathcal L_a V(x)\le -\lambda p V(x)+Mp
    \]
    up to lower order constants absorbed into \(M\).
\end{enumerate}

The following is a standard result of the preceding conditions.
\begin{proposition}[Stationary measure and moments]\label{prop:stationary_fpk}
If \(a\in\mathcal A(\nu,p_0,\lambda,M)\), then the stationary Fokker--Planck equation
\(\mathcal L_a^*\mu=0\) has a unique probability solution \(\mu_a\). Moreover, for every \(p\le p_0\), there is a
constant \(M_p<\infty\), depending only on
\((\nu,p_0,\lambda,M,p)\), such that
\[
\sup_{a\in \mathcal{A}(\nu,p_0,\lambda,M)}\int |x|^p\,\mu_a(dx)\le M_p,
    \qquad 2\le p\le p_0.
\]
\end{proposition}

Classical sufficient conditions for existence and uniqueness of stationary
solutions to the Fokker--Planck equation are given in
\cite[Theorems~2.4.1 and~4.1.6]{bogachev2022fokker}. The moment estimate
follows from the standard Lyapunov cutoff argument; the details are given in
Appendix~\ref{app:proofs}.

\begin{remark}
Local Lipschitz conditions on \(a=(b,\sigma)\) are often imposed for strong
well-posedness of the SDE. The present theory is stated at the level of laws
and stationary Fokker--Planck equations, so one may instead work with weak solutions and
the associated martingale problem. In this setting, it does not require the conditions on the derivatives of $a$. See also  the Stroock-Varadhan uniqueness theorem \cite[Theorem 24.1, Chapter V]{rogers2000diffusions}.
\end{remark}

\begin{lemma}[Closedness and stability of admissible coefficients]\label{lemma:inv_stability}
The set \(\mathcal A(\nu,p_0,\lambda,M)\) is closed under the compact-open
topology. That is, if \(a_n\in\mathcal A(\nu,p_0,\lambda,M)\) and
\(a_n\to a\) locally uniformly, then \(a\in\mathcal A(\nu,p_0,\lambda,M)\).
Moreover,
\[W_2(\mu_{a_n},\mu_a)\to0 .\]
\end{lemma}
The proof of Lemma \ref{lemma:inv_stability} will also be presented in Appendix \ref{app:proofs}.

\subsection{An explicit construction of approximating velocity field}

We will basically use the smoothed flow matching velocity field to construct an approximating velocity field.
First let \(P_S:\mathbb R^d\to\overline B_S\) be the radial projection and define
\[\mu_a^S=(P_S)_\#\mu_a .\]
The polynomial moment bound gives the projection error
\[
\sup_{a\in \mathcal{A}(\nu,p_0,\lambda,M)}W_2(\mu_a^S,\mu_a)
\le M_{p_0}^{1/2}S^{-(p_0-2)/2}.
\]
Moreover, \(a\mapsto\mu_a^S\) is continuous.

Next, take \(\alpha\in(0,1)\) and the narrower Gaussian
\[\gamma_\alpha=\mathcal N(0,(1-\alpha)I_d).\]
Fix the target measure \(\mu_a^S\) with \(a\in\mathcal{A}(\nu,p_0,\lambda,M)\). Draw
\[X\sim\gamma_\alpha,\qquad Y\sim\mu_a^S,\qquad\xi\sim\mathcal N(0,I_d),\]
independently. Define
\[M_t^a=(1-t)X+tY,\qquad Z_t^{a}=M_t^a+\sqrt\alpha\,\xi .\]
Clearly, letting \(\rho_t^{a}=\operatorname{Law}(Z_t^{a})\), then
\[\rho_0^{a}=\gamma, \qquad \rho_1^{a}=\mu_a^S*\mathcal N(0,\alpha I_d).\]
Coupling \(Y\) with \(Y+\sqrt\alpha\,\xi\) gives
\[
W_2^2(\mu_a^S*\mathcal N(0,\alpha I_d),\mu_a^S)
\le
\mathbb E|\sqrt\alpha\,\xi|^2
=\alpha d \quad \Longrightarrow \quad  W_2(\rho_1^{a},\mu_a^S)\le \sqrt{\alpha d}.
\]
Let
\[\phi_\alpha(u)=(2\pi\alpha)^{-d/2}\exp(-|u|^2/(2\alpha)).\]
Then, the smoothed density is
\[
\rho_t^{a}(z) =\int \phi_\alpha(z-((1-t)x+ty))\,
\gamma_\alpha(dx)\,\mu_a^S(dy)>0,\quad \forall z,\ t\in [0,1].
\]
The current is
\[J_t^{a}(z)=\int (y-x)\phi_\alpha(z-((1-t)x+ty))\,\gamma_\alpha(dx)\,\mu_a^S(dy).\]
The velocity is
\[v^a(z,t)=\frac{J_t^{a}(z)}{\rho_t^{a}(z)}=\mathbb E[Y-X\mid Z_t^{a}=z].\]
\begin{lemma}[Smoothed flow matching velocity]\label{lem:smoothed_velocity}
For fixed \(\alpha\in(0,1)\) and \(S<\infty\), the velocity \(v^a\) has
uniform path moment bounds: for every \(p\ge2\),
\[
    \sup_{a\in \mathcal{A}(\nu,p_0,\lambda,M)}\sup_{t\in[0,1]}
    \int |v^a(z,t)|^p\,d\rho_t^{a}(z)
    \le C_{\alpha,S,p}.
\]
It also has the uniform linear growth bound
\[|v^a(z,t)|\le C_{\alpha,S}(1+|z|).\]
Moreover, \(\rho_t^{a}\) satisfies the continuity equation with this velocity
field,
\begin{gather}
    \partial_t \rho_t^{a} + \nabla\cdot(\rho_t^{a}v^a)=0.
\end{gather}
\end{lemma}

Finally, we take a smooth cutoff \(\chi_R\in C_c^\infty(\mathbb R^d)\) such that
\[
\chi_R(x)=1 \text{ for } |x|\le R,
\qquad \chi_R(x)=0 \text{ for } |x|\ge 2R, \qquad 0\le \chi_R\le1 .
\]
For a fixed \(a\in\mathcal{A}(\nu,p_0,\lambda,M)\), define
\[v^{a,R}(x,t)=\chi_R(x)v^a(x,t).\]
\begin{lemma}[Lipschitz continuity of the target velocity]\label{lem:compact_target_continuity}
For every fixed \(\alpha\in(0,1)\), \(S<\infty\), and \(R>0\),
\(a\mapsto v^{a,R}\) is continuous as a map from \(\mathcal{A}(\nu,p_0,\lambda,M)\) to
\(C(\overline B_{2R}\times[0,1];\mathbb R^d)\). Moreover,
\[
L_{\alpha,S,R}
=\sup_{a\in \mathcal{A}(\nu,p_0,\lambda,M)}\sup_{t\in[0,1]}
\operatorname{Lip}_x v^{a,R}(\cdot,t)<\infty .
\]
\end{lemma}

Let \(\mu_a^{S,R}\) be the terminal law generated by
\[\dot X_t=v^{a,R}(X_t,t),\qquad X_0\sim\gamma .\]
\begin{lemma}[Truncation consistency]\label{lem:truncation_consistency}
For every fixed \(\alpha\in(0,1)\) and \(S<\infty\),
\[
\varepsilon_{\mathrm{trun}}(\alpha,S,R)
:=\sup_{a\in \mathcal{A}(\nu,p_0,\lambda,M)}W_2(\rho_1^a,\mu_a^{S,R})
\le C_{\alpha,S,p}R^{-(p-2)/2}.
\]
\end{lemma}

The proofs of Lemmas~\ref{lem:smoothed_velocity}--\ref{lem:truncation_consistency}
are given in Appendix~\ref{app:canonical_target}.

\subsection{Expressivity of the neural operator approximation}

Now, we will consider a compact subset $K$ of $\mathcal{A}(\nu, p_0, \lambda, M)$ for some parameters $\nu, p_0, \lambda, M$.
By Proposition~\ref{prop:compactness_criterion}, a closed set $K\subset \mathcal{A}(\nu, p_0, \lambda, M)$ is compact if the following hold:
(i) for each ball \(B_m=\{x:|x|\le m\}\), the restrictions of $a\in K$ on $B_m$ are
  uniformly bounded; (ii) for each \(B_m\), the restrictions of $a\in K$ on $B_m$ are equicontinuous.
Lemma~\ref{lemma:inv_stability} shows that $\mathcal{A}(\nu, p_0, \lambda, M)$ is closed in $\mathcal{X}_a$. Hence the closure in $\mathcal{X}_a$ of any family in $\mathcal{A}(\nu, p_0, \lambda, M)$ that satisfies (i) and (ii) is such a set $K$.

Below, we put the following assumption on the neural operator architecture.
We will then discuss how the assumption is satisfied for the architecture we use in the next section.
\begin{assumption}[Compact domain neural operator approximation]\label{ass:operator_uat}
For every compact set $K$, every fixed envelope constant \(C_{\rm env}<\infty\), every \(R>0\),
every continuous operator
\[\Psi:K\to C(\overline B_{2R}\times[0,1];\mathbb R^d)\]
with \(|\Psi(a)(x,t)|\le C_{\rm env}(1+|x|)\), and every \(\delta>0\), the following hold.
\begin{enumerate}[(a)]
\item The
model class contains a parameter \(\theta\) such that
\[
    \sup_{a\in K}
    \sup_{(x,t)\in\overline B_{2R}\times[0,1]}
    |\mathcal N_\theta(a)(x,t)-\Psi(a)(x,t)|<\delta .
\]
\item The selected realizations are locally Lipschitz in the spatial variable (i.e., Lipschitz on
compact sets) and satisfy the global linear growth envelope for constant $C$ depending on the chosen parameters such that
\[
    |\mathcal N_\theta(a)(x,t)|
    \le C(1+|x|),
    \qquad a\in K,\ x\in\mathbb R^d,\ t\in[0,1].
\]
\end{enumerate}
\end{assumption}

Assumption~\ref{ass:operator_uat} is the ordinary compact domain universal
approximation property (UAP) for the neural operator. The reduction theorem uses the
Lipschitz constant of the target velocity \(v^{a,R}\), not the
Lipschitz constant of the approximating network. The global linear growth
clause is only used to make the exact learned ODE globally well posed and to
control the probability that the learned trajectory leaves the compact
approximation domain.

Next, one needs the following auxiliary results, whose proofs are given in Appendix \ref{app:express}.
\begin{proposition}[ODE stability for a Lipschitz target]\label{prop:flow_stability}
Let \(v_t\) and \(\hat v_t\) generate flows \(\Phi_t\) and \(\hat\Phi_t\)
from the same initial law \(\rho_0\). Assume \(v_t\) is globally Lipschitz in
\(x\), uniformly in \(t\), with constant \(L\), and
\[
    \sup_{x,t}|v_t(x)-\hat v_t(x)|\le\delta .
\]
Then
\[
    W_2((\Phi_1)_\#\rho_0,(\hat\Phi_1)_\#\rho_0)
    \le e^L\delta .
\]
\end{proposition}

\begin{proposition}[Reduction from compact operator approximation to sampling error]\label{thm:reduction}
Let $K\subset \mathcal{A}(\nu, p_0, \lambda, M)$ be a compact set of coefficients $a=(b, \sigma)$ for some positive parameters $\nu>0, p_0>2, \lambda>0, M>0$. Fix
\(\alpha\in(0,1)\), \(S<\infty\), \(R>0\), and \(p>2\). The implemented
velocity field is the plain neural operator
\[
    \hat v_{\theta}[a](x,t) :=  \mathcal N_\theta(a)(x,t).
\]
Assume \(\mathcal N_\theta(a)\) is locally Lipschitz in \(x\), and satisfies
\[
    |\mathcal N_\theta(a)(x,t)|  \le C_{\alpha,S}(1+|x|),
    \qquad a\in K,\ x\in\mathbb R^d,\ t\in[0,1],
\]
and define
\[
    \delta_{\theta,\alpha,S,R}  :=   \sup_{a\in K}  \sup_{(x,t)\in\overline B_{2R}\times[0,1]}   |\mathcal N_\theta(a)(x,t)-v^{a,R}(x,t)|.
\]
If $\delta_{\theta,\alpha,S,R}<Re^{-L_{\alpha,S,R}}$
and \(\hat\mu_{a,\theta}\) is the terminal law generated by
\(\hat v_{\theta}[a]\) from \(X_0\sim\gamma\), then
\[
    \sup_{a \in K}W_2(\hat\mu_{a,\theta},\mu_a)
    \le e^{L_{\alpha,S,R}}\delta_{\theta,\alpha,S,R} +C_{\alpha,S,p}R^{-(p-2)/2}
    +\varepsilon_{\mathrm{trun}}(\alpha,S,R)
    +\sqrt{\alpha d}
    +M_{p_0}^{1/2}S^{-(p_0-2)/2}.
\]
\end{proposition}

The five terms have distinct meanings. The first is the compact domain
neural operator approximation error. The second is the escape tail error
introduced by using the exact, unprojected learned ODE while only controlling
the approximation on \(B_{2R}\). The third is the ODE error caused by
truncating the target velocity outside \(B_R\). The fourth is the Gaussian
mollification bias, and the fifth is the error from projecting the invariant
measure to \(B_S\).

The following then gives the main claim about expressivity of the model, and the proof is provided in Appendix \ref{app:express}.
\begin{theorem}[Ordinary UAP implies sampler expressivity]\label{cor:conditional_expressivity}
Let $K\subset \mathcal{A}(\nu, p_0, \lambda, M)$ be a compact set for some positive parameters $\nu>0, p_0>2, \lambda>0, M>0$. Suppose Assumption \ref{ass:operator_uat} holds. Then for
every \(\eta>0\) there exist \(S<\infty\), \(\alpha\in(0,1)\), \(R>0\), and a
parameter \(\theta\) such that
\[
    \sup_{a\in K}
    W_2(\hat\mu_{a,\theta},\mu_a)<\eta .
\]
\end{theorem}

\subsection{Universal approximation property of the Lagrangian attention neural operator}

In this subsection, we establish sufficient conditions for the Lagrangian attention neural operator to satisfy the approximation assumption used in the preceding expressivity theorem.
The key is that the finite probes must resolve the coefficient family.
We show that this holds with high probability under the stated assumptions on the probe distribution.


For the Lagrangian sensors, we do not need the cutoff $J$ as for the grid samplers.
We recall that  a Lagrangian input is the finite unordered set of probe
observations of Section~\ref{sec:perceiver},
\[
    \mathcal S_{m,L}(a)=\{Y_1^a,\ldots,Y_m^a\},
    \qquad
    Y_i^a=(y_{i,j}^a)_{j=0}^L\in\mathbb R^{(L+1)\times(2d+dd_W)},
\]
with the initial probe locations, probe noise, step size, and trajectory
length \(L\) fixed before evaluation. The realized operator
\(\mathcal N_\theta(a)=v_\theta[a]\) is the trajectory CNN encoder, the
finite depth Perceiver aggregator \(\operatorname{Encoder}\) with readout
\(z\in\mathbb R^q\), and the AdaLN-conditioned residual decoder
\(\operatorname{Decoder}(z,x,t)\).

\begin{proposition}[Finite sensor Lagrangian Perceiver--AdaLN approximation]\label{prop:lagrangian_uat}
Let \(Y=\overline B_{2R}\times[0,1]\), and fix an envelope constant
\(C_{\rm env}<\infty\).  Assume:
\begin{enumerate}[(a)]
\item (Finite Sensor Resolution, FS) Finite probes resolve \(K\): for every
\(\varepsilon>0\) there are \(m,L\), fixed probe seeds, and \(r>0\) such that
\[
    d_{\rm sens}(\mathcal S_{m,L}(a),\mathcal S_{m,L}(a'))<r
    \;\Longrightarrow\;
    d_{\mathcal X_a}(a,a')<\varepsilon,
    \qquad a,a'\in K,
\]
where \(d_{\rm sens}\) is any distance on finite sensor sets.
\item (Encoder-decoder Universality, EU) On the compact image \(\mathcal S_{m,L}(K)\), the encoder-decoder architecture class uniformly approximates any continuous map invariant under sensor permutations
\(F:\mathcal S_{m,L}(K)\times Y\to\mathbb R^d\) whose realizations are
locally Lipschitz in \(x\). 
\end{enumerate}
Then the finite sensor encoder-decoder class satisfies Assumption~\ref{ass:operator_uat} (a).
 If the activation function is globally Lipschitz, then
Assumption~\ref{ass:operator_uat} (b) is also satisfied.
\end{proposition}

The proof is given in Appendix~\ref{app:express}.

We remark that a natural choice is 
\[
d_{\rm sens}(\{y_i\},\{y'_i\})=\max_{ij}\|a(x_{ij})-a'(x_{ij})\|.
\]
 Any distance that makes \(\mathcal S(K)\) compact and metrizes the topology of the sensor set suffices.

Next, we discuss the validity of the conditions used. We first verify that the condition (FS) holds with large probability as summarized in Theorem \ref{thm:randomFS}. The proof is given in Appendix \ref{app:express}.
\begin{theorem}\label{thm:randomFS}
Suppose that the initial points $\{x_{i0}\}_{i=1}^m$ are sampled from some
distribution $\nu_0(dx)=\rho_0(x)\,dx$ with a continuous density $\rho_0(x)>0$ that is positive everywhere.
For any $\delta\in (0, 1)$, there is $m_0=m(\delta,r,\varepsilon)>0$ such that when $m>m_0$, the Lagrangian attention neural operator satisfies the condition (FS) with probability
$1-\delta$. Consequently, provided the architecture satisfies the ordinary
finite dimensional UAP on the compact set with permutation invariance (i.e., the (EU) property), the Lagrangian attention neural operator
has the expressivity with probability at least $1-\delta$.
\end{theorem}

\begin{remark}
We remark that the estimate of $m$ in the proof of Theorem \ref{thm:randomFS} in Appendix \ref{app:express}
has curse of dimensionality. The proof in Appendix \ref{app:express} is based on the $h$-net and does not take advantage of the Monte Carlo feature of the method, so it should have overestimated $m$. The improvement needs a better topology for the coefficient functions that could make use of the Monte Carlo rate. This is left for future study.
\end{remark}

Next, let us discuss about universality condition and provide some evidence.
 Since \(m\) and \(L\) are fixed, the sensor set \(\{e_1,\ldots,e_m\}\subset\mathbb R^{d_e}\)
lives in the compact finite dimensional space \(\mathbb R^{m\times d_e}\), so (EU) is
an ordinary UAP for continuous functions on a compact set---the only nonstandard
requirement is permutation invariance in the sensor index. Three existing results
together make (EU) reasonable for the present architecture:
\begin{itemize}
\item \emph{Cross attention is a universal approximator.}
Liu et al.~\cite{liu_attention_uat_2025} prove that a single-head cross attention
layer, preceded by a linear projection, universally approximates any continuous
function on a compact domain in \(L^\infty\). The cross attention block
\(\mathrm{CA}(Z,E)\) of our encoder fits this setting directly: the learned
initial latent \(Z^{(0)}\) plays the role of the query and the sensor array \(E\)
provides the keys and values.
\item \emph{Self attention\,+\,FFN is a universal approximator of equivariant maps.}
Yun et al.~\cite{yun_transformers_2020} prove that a transformer (multihead
self attention followed by a positionwise FFN) universally approximates any
continuous permutation-equivariant sequence-to-sequence function on a compact domain.
The latent self attention and FFN blocks in each encoder layer are exactly this
subarchitecture, applied to the fixed length latent array of \(k\) tokens.
\item \emph{Permutation invariance is preserved by construction.}
Cross attention is invariant to permutation of the keys and values (the sensor
embeddings) because each attention weight is computed by an inner product between a
fixed query and each key independently before softmax normalization; permuting the
sensors permutes the weights but leaves the weighted sum unchanged. The subsequent
self attention operates on the latent array, which has no sensor index, and the
linear readout \(W^{\rm out}\operatorname{vec}(Z^{(L_{\rm enc})})\) is fixed dimension
and sensor-order-independent. Our encoder is structurally identical to the
Set Transformer of~\cite{lee_set_2019}, which was introduced precisely for
set aggregation that is invariant under permutations.
\end{itemize}
A formal proof combining all three stages in a single theorem statement does not yet
exist in the literature; we treat (EU) as an assumption, supported by the above
components. The decoder, given a fixed \(z\) and compact \(Y\), satisfies the
standard MLP UAP.


\section{Resolution Invariance}\label{sec:resoinvar}

This section is about the resolution of the input function. It explains when changing the sensor resolution leaves the realized sampler stable, in expectation over the probe
trajectories, for the fixed trained attention architecture.  A finite sensor
Lagrangian Perceiver--AdaLN approximation statement records sufficient
conditions for the compact domain approximation assumption once the probes
resolve the coefficient family.   The resolution invariance thus allows the model to be trained on data samples with variable resolution, and be applied for different resolution for new unseen cases.  For random probes, the probes are sampled in the
whole state space according to a prescribed probe law, so the resolution
parameter is the number \(m\) of probes.

In this subsection, \(K\subset\mathcal X_a\) denotes a compact
coefficient family.  Throughout this subsection the weights \(\theta\) are fixed and
suppressed from the notation.

For a fixed full space probe recipe, one probe produces a random observation
\(Y\) as introduced in Section~\ref{sec:perceiver}---the visited states
together with the drift \(b\) and diffusion \(\sigma\) recorded along a short
probe path---with CNN embedding \(e=\mathrm{CNN}(Y)\in\mathbb R^{d_e}\) and
law of the sensor embedding
\[
\nu_a =\operatorname{Law}(e).
\]
With \(m\) independent probes \(\mathbf e=(e_1,\ldots,e_m)\),
\(e_i\stackrel{\mathrm{i.i.d.}}{\sim}\nu_a\), the empirical law of the sensor embeddings is
\[
\hat\nu_{a,m}=\frac1m\sum_{i=1}^m\delta_{e_i}.
\]
Because the sensor order is irrelevant, we identify the encoded sensor set
with its empirical measure.

The encoder reads the sensors \emph{only} through cross attention, and only through
the projected sensor \(\tilde e=eP_e\in\mathbb R^{d_z}\), with \(P_e\) the shared
projection of Section~\ref{subsec:perceiver_encoder}. Replacing the sensor rows of
\(\mathrm{Atten}\) by a measure \(\nu\) on the embeddings, a head with query \(q\) is
\[
    A^h_\ell(q,\nu)
    =\frac{\int F^h_\ell(q,e)\,d\nu(e)}{\int \omega^h_\ell(q,e)\,d\nu(e)},
    \quad
    \omega^h_\ell(q,e)=\exp\!\Bigl(\frac{\langle q,\,\tilde eW^K_{\ell,h}\rangle}{\sqrt{d_h}}\Bigr),
    \quad
    F^h_\ell(q,e)=\omega^h_\ell(q,e)\,\tilde eW^V_{\ell,h}\in\mathbb R^{d_h},
\]
which at \(\nu=\hat\nu_{a,m}=\frac1m\sum_i\delta_{e_i}\) reduces to the
\(\mathrm{Atten}\) head over the \(m\) sensors.

We now use this measure formulation to write the full encoder recursion. We use \(H_\ell\) to denote the
cross attention message and \(G_\ell\) the self attention--MLP update. With token queries
\(q=z_rW^Q_{\ell,h}\), the per-head outputs assembled by \(W^O_\ell\) give
\(H_\ell(Z,\nu)\in\mathbb R^{k\times d_z}\), with rows
\begin{gather}
    \bigl[H_\ell(Z,\nu)\bigr]_r=\bigl[\,A^h_\ell(z_rW^Q_{\ell,h},\nu)\,\bigr]_{h=1}^{H}\,W^O_\ell,
    \quad r=1,\dots,k .
\end{gather}

Here, $G_{\ell}$ represents the transform made by the self attention and the MLP
\[
    G_\ell(Y)=\mathrm r[\mathrm{MLP}]\bigl(\mathrm r[\mathrm{SA}](Y)\bigr),
\]
where $\mathrm r[f](Y)=\mathrm{LN}\bigl(Y+f(Y)\bigr)$ is
a residual sublayer followed by LayerNorm.
The latents are then transformed by the self attention and the MLP through
  the depth-\(L_{\rm enc}\) recursion:
\begin{equation}\label{eq:Zrecursion}
    Z^{(\ell+1)}=G_\ell\bigl(\mathrm{LN}(Z^{(\ell)}+H_\ell(Z^{(\ell)},\nu))\bigr),\quad
    \ell=0,\dots,L_{\rm enc}-1,
\end{equation}
from the learned \(Z^{(0)}\).

Hence, the encoder is then a functional of the measure 
\[
\operatorname{Encoder}(\nu):=W^{\rm out}\operatorname{vec}\bigl(Z^{(L_{\rm enc})}\bigr).
\]
Denote 
\[
    z_{a,m}=\operatorname{Encoder}(\hat\nu_{a,m}),
    \qquad
    z_a=\operatorname{Encoder}(\nu_a)
\]
the encodings for the empirical embedded sensor law, and the population embedded sensor law respectively.

The finite and
infinite sensor velocity fields are the decoder pushforwards
\[
    v_{a,m}(x,t)=\operatorname{Decoder}(z_{a,m},x,t),
    \qquad
    v_a(x,t)=\operatorname{Decoder}(z_a,x,t).
\]
The corresponding terminal laws are the random measure \(\hat\mu_{a,m}\) (it
depends on the specific random realized probes \(\mathbf e\)) and the limit measure
\(\hat\mu_a\) corresponding to the continuum probe measure:
\[
    \hat\mu_{a,m}
    =
    \operatorname{Law}(X_1^{a,m}),
    \qquad
    \dot X_t^{a,m}
    =
    v_{a,m}(X_t^{a,m},t),
    \qquad
    X_0^{a,m}\sim\gamma,
\]
and
\[
    \hat\mu_a
    =
    \operatorname{Law}(X_1^{a}),
    \qquad
    \dot X_t^{a}
    =
    v_a(X_t^{a},t),
    \qquad
    X_0^{a}\sim\gamma .
\]

For each probe realization, these definitions give the following finite sensor map.

\[
a=(b,\sigma)
\longmapsto
\mathcal S_{m,L}(a)
\longmapsto
\hat\nu_{a,m}
\longmapsto
z_{a,m}
\longmapsto
v_{a,m}
\longmapsto
\hat\mu_{a,m}.
\]

\begin{definition}[Random sensor resolution invariance]
The Lagrangian Perceiver sampler is random sensor resolution invariant on
\(K\) if the population-probe terminal law
\(\hat\mu_a\) is defined for each \(a\in K\) and
\[
    \lim_{m\to\infty}
    \sup_{a\in K}
    \mathbb E_{\mathbf e}\,
    W_2(\hat\mu_{a,m},\hat\mu_a)=0 ,
\]
where the expectation is over the random probe trajectories. 
\end{definition}

\begin{remark}
The law
\(\hat\mu_a\) is the continuum limit only in the sense that
the empirical probe measure has been replaced by the underlying probe law in
the same finite depth Perceiver architecture; the probe recipe itself is fixed.
\end{remark}

We now state the two structural assumptions. Both are mild and realized by the architecture; the quantitative bounds the
proof uses---boundedness and Lipschitz continuity of the attention weights, a
strictly positive denominator, compact reachable latents, and a Lipschitz
decoder---are then \emph{consequences} (Lemma~\ref{lem:perceiver_consequences}),
not separate hypotheses.

\begin{assumption}[Compact sensor embeddings and a regular decoder]\label{ass:perceiver}
The following hold uniformly over \(a\in K\):
\begin{enumerate}
\item[(A1)] \textbf{Compact embedding support.} There is a fixed compact set
\(K_e\subset\mathbb R^{d_e}\) with \(\operatorname{supp}\nu_a\subseteq K_e\) for every
\(a\in K\). 

\item[(A2)] \textbf{Decoder state regularity.} There is \(C_\theta<\infty\), and for
every \(R<\infty\) an \(L^x_{\theta,R}<\infty\), such that for all \(a\in K\),
every reachable latent \(z\), all \(t\in[0,1]\) and \(x,y\in\overline B_R\),
\[
    |\operatorname{Decoder}(z,x,t)|\le C_\theta(1+|x|),
    \qquad
    |\operatorname{Decoder}(z,x,t)-\operatorname{Decoder}(z,y,t)|\le L^x_{\theta,R}|x-y| .
\]
\end{enumerate}
\end{assumption}

Let us make some discussions on the assumptions above. Regarding (A1), the per-sensor embedding of Section~\ref{sec:perceiver} pools the final
convolutional block, \(e_i=\frac{1}{L+1}\sum_{j=0}^{L}C^{(N)}_{j,:}\) with
\(C^{(N)}=\mathrm{GELU}(\mathrm{LayerNorm}(\cdots))\). The LayerNorm normalizes each
row to Euclidean norm \(\le\sqrt{d_e}\) and applies its fixed affine map, so its output
rows are bounded by some \(\rho_e<\infty\); since \(|\mathrm{GELU}(t)|\le|t|\), every
\(\|C^{(N)}_{j,:}\|\le\rho_e\), and averaging gives \(\|e_i\|\le\rho_e\). Hence
\(\operatorname{supp}\nu_a\subseteq K_e:=\overline B_{\rho_e}\), a fixed compact set
independent of the probe.

Regarding (A2),   in the decoder of Section~\ref{sec:perceiver}, \(x\) enters only through
\(h^{(0)}=W_{\mathrm{in}}[x;\tau]+b_{\mathrm{in}}\); the AdaLN blocks and the output
map \(v=W_v\mathrm{LN}(h^{(N_d)})\) compose the affine maps
\(W_{\mathrm{in}},W_1,W_2,W_v\), the activation \(\mathrm{Mish}\)
(\(L_{\mathrm{act}}:=\operatorname{Lip}(\mathrm{Mish})\)), and \(\varepsilon\)-LayerNorms
(\(L_{\mathrm{LN}}:=\operatorname{Lip}(\mathrm{LN})\)), with the scales
\(1+\gamma(c),\alpha(c)\) applied multiplicatively. Each block is thus Lipschitz in
its input with constant
\[
    1+\|\alpha(c)\|_\infty\,\|W_1\|\,\|W_2\|\,L_{\mathrm{act}}^2\,L_{\mathrm{LN}}^2\,\|1+\gamma(c)\|_\infty^2 ,
\]
so \(x\mapsto\operatorname{Decoder}(z,x,t)\) has Lipschitz constant at most the product
of \(\|W_{\mathrm{in}}\|\), these \(N_d\) block constants, and \(\|W_v\|L_{\mathrm{LN}}\).
As \(\gamma,\alpha\) are linear and \(c=W_c[z;\tau]+b_c\) ranges over the compact image
of \(K_z\times[0,1]\), the scales \(\|1+\gamma(c)\|_\infty,\|\alpha(c)\|_\infty\) are
bounded there, so that product is a finite \(\Lambda\) uniform over
\(K_z\times[0,1]\), giving (A2) with \(L^x_{\theta,R}=\Lambda\) and
\(C_\theta=\sup_{z\in K_z,\,t}|\operatorname{Decoder}(z,0,t)|+\Lambda\).

\begin{lemma}[Consequences of compactness]\label{lem:perceiver_consequences}
Under \textnormal{(A1)} and the fixed smooth encoder of
Section~\ref{subsec:perceiver_encoder}, whose blocks end in
\(\varepsilon\)-regularized LayerNorm. Then, the following hold.
\begin{enumerate}
\item[(i)] there are constants \(0<\omega_{\min}\le\omega_{\max}<\infty\) and $M_F>0$ such that
\[
\omega_{\min}\le\omega^h_\ell(q,e)\le\omega_{\max}
\]
 and
\(\|F^h_\ell(q,e)\|\le M_F\); consequently
\(\int\omega^h_\ell(q,\cdot)\,d\nu\ge\omega_{\min}\) for \emph{every} probability
measure \(\nu\) on \(K_e\), in particular for both \(\hat\nu_{a,m}\) and \(\nu_a\);

\item[(ii)] there is $L_{\omega}>0$ such that
\[
|\omega^h_\ell(q,e)-\omega^h_\ell(q',e)|
+\|F^h_\ell(q,e)-F^h_\ell(q',e)\|\le L_\omega\|q-q'\|;
\]

\item[(iii)] There are two compact sets \(K_z,K_q\) such that
every reachable latent array lies in \(K_z\) and every query in
\(K_q\); consequently, the self attention--MLP update
\(G_\ell\) is Lipschitz on the reachable compact set, with constants
\(L^G_\ell\);
\item[(iv)] for every \(R<\infty\) the decoder is Lipschitz in \(z\) on
\(K_z\times\overline B_R\times[0,1]\): there is \(L^{\rm dec}_{\theta,R}<\infty\) with
\[
|\operatorname{Decoder}(z,x,t)-\operatorname{Decoder}(z',x,t)|\le L^{\rm dec}_{\theta,R}\|z-z'\|.
\]
\end{enumerate}
\end{lemma}

These bounds yield the following uniform law of large numbers (ULLN).
Both proofs are in Appendix~\ref{app:randomresolution}.

\begin{lemma}[ULLN for the attention class]\label{lem:attention_ulln}
Over the finitely many layers, heads, and value coordinates, set
\[
    \mathfrak F_\theta
    =\bigcup_{\ell,h}\Bigl(\{\omega^h_\ell(q,\cdot):q\in K_q\}
    \cup\{[F^h_\ell(q,\cdot)]_r:q\in K_q,\ r=1,\dots,d_h\}\Bigr),
\]
a uniformly bounded family of functions on \(K_e\), and
\[
    \delta_m^{\rm att}
    :=\sup_{a\in K}\mathbb E_{\mathbf e}
      \sup_{f\in\mathfrak F_\theta}
      \Bigl|\int f\,d\hat\nu_{a,m}-\int f\,d\nu_a\Bigr| .
\]
Under Assumption~\ref{ass:perceiver}, with \(\mathbf e=(e_1,\dots,e_m)\)
i.i.d.\ from \(\nu_a\) and \(\operatorname{supp}\nu_a\subseteq K_e\),
\(\delta_m^{\rm att}\to0\).
\end{lemma}

\begin{remark}
The limit \(\delta_m^{\rm att}\to0\) needs only total boundedness of
\(\mathfrak F_\theta\). A rate is available but is not used below: by
Lemma~\ref{lem:perceiver_consequences}(ii) the map \(q\mapsto f_q\) is Lipschitz from
the compact \(K_q\) into \((C(K_e),\|\cdot\|_\infty)\), so \(\mathfrak F_\theta\) has
finite metric entropy
\(\log N(\varepsilon,\mathfrak F_\theta,\|\cdot\|_\infty)\lesssim\dim(q)\,\log(1/\varepsilon)\),
and the standard Dudley/symmetrization bound then gives
\(\delta_m^{\rm att}\le C_\theta m^{-1/2}\) with \(C_\theta\) independent of \(a\)
(the class and its envelope \(M_\theta=\sup_{f\in\mathfrak F_\theta}\|f\|_\infty\) are
distribution independent).
\end{remark}

Finally, we obtain the following resolution invariance result for the random sensor attention sampler.
The proof is given in Appendix \ref{app:randomresolution}.
\begin{theorem}\label{prop:random_sensor_di}
Under Assumption~\ref{ass:perceiver}, the population-probe terminal law \(\hat\mu_a\)
is defined for each \(a\in K\), and
\[
\lim_{m\to\infty}\sup_{a\in K}
    \mathbb E_{\mathbf e}\,  W_2(\hat\mu_{a,m},\hat\mu_a)=0.
\]
so the Lagrangian Perceiver sampler is random sensor resolution invariant on
\(K\) \emph{in expectation over the probes}. 
\end{theorem}

We would like to point out that the limit \(\hat\mu_a\) is the
population-probe law of the same Perceiver with fixed weights, not the target invariant
measure \(\mu_a\).


\section{Numerical Experiments}\label{sec:experiments}


We perform numerical tests to validate our framework. The necessary setup and results needed to support our conclusions are given here while leave other details (data generation, architecture, training, software etc) are in
Appendix~\ref{app:expdetails}.

To compare generated and reference samples, we use the Sinkhorn divergence, a bias-corrected, entropy-regularized approximation to quadratic optimal transport~\cite{feydy_sinkhorn_2019}. For their empirical probability
measures $\mu$ and $\nu$, define
\begin{align}
\mathrm{OT}_{\varepsilon}(\mu,\nu)
&= \inf_{\pi\in\Pi(\mu,\nu)}
\left\{\frac12\int |x-y|^2\,d\pi(x,y)
+\varepsilon\,\mathrm{KL}(\pi\,\|\,\mu\otimes\nu)\right\},\nonumber\\
S_{\varepsilon}(\mu,\nu)
&= \mathrm{OT}_{\varepsilon}(\mu,\nu)
-\frac12\mathrm{OT}_{\varepsilon}(\mu,\mu)
-\frac12\mathrm{OT}_{\varepsilon}(\nu,\nu).
\label{eq:sinkhorn_divergence}
\end{align}
Here $\mathrm{KL}$ denotes relative entropy, and $\varepsilon>0$ controls
the amount of smoothing. $S_\epsilon$ is the Sinkhorn divergence. The self-comparison terms remove the entropic
bias and give $S_{\varepsilon}(\mu,\mu)=0$. 
$S_{\varepsilon}(\mu,\nu)$ converges to
$\tfrac12 W_2^2(\mu,\nu)$ as $\varepsilon\downarrow0$.
We set $\varepsilon=0.0025$ in the experiments, and report $S_{\varepsilon}$ without a square root, so the Sinkhorn and $W_2$ columns have different scales. Lower values indicate closer agreement. Sinkhorn divergence is easier to compute and more stable then $W_2$ in high-dimensional spaces.

\subsection{\expVariableNoise{} and \expRareEvent}\label{sec:oned_experiments}

We first test the amortized neural sampler in \expVariableNoise{} and
\expRareEvent, using the
simple grid Multi-input DeepONet described at the beginning of
Section~\ref{sec:perceiver}. These experiments show that operator learning of
invariant measures is feasible, and locate the regime---slow mixing---where
it gains an advantage over MCMC. Both experiments are trained on
65{,}536 function instances and evaluated on 1{,}024 test functions,
reporting the Sinkhorn divergence and the 2-Wasserstein distance ($W_2$)
against ground truth; the evaluation protocol and the shared architecture and
training configuration are given in Appendix~\ref{app:oned_details}
(Table~\ref{tab:1d_params}).

\subsubsection{\expVariableNoise}\label{sec:variablenoise}

This \expVariableNoise{} example serves as a sanity check in a benign, fast mixing regime. MCMC
performs well here. The experiment verifies that the amortized operator sampler
produces reasonable samples across a large family of SDEs and establishes a
baseline for comparison with the harder rare event setting that follows.
We set $b = -\nabla w + s$ and $\sigma = \max\{0.25,\, 1+s'\}$, where
$w(x) = {(\max\{|x|-2,0\})}^2$ is a fixed confining well and the lower bound on
$\sigma$ ensures uniform ellipticity.  The perturbations $s,s'$ (functions of $x$) are independent
draws from a Gaussian random field with squared exponential kernel (length
scale~$1$, variance~$1$).

\begin{table}[htbp]
\centering
\caption{\expVariableNoise. Distribution metrics on 1{,}024 test functions
  (mean\;/\;median).}
\label{tab:1d_vn}
\begin{tabular}{lccc}
\toprule
 & Sinkhorn & $W_2$ & Time (s) \\
\midrule
\textbf{Ours} & $0.018\;/\;0.002$ & $0.204\;/\;0.125$ & $19.0$ \\
MCMC & $\mathbf{0.001}\;/\;\mathbf{0.0002}$ & $\mathbf{0.083}\;/\;\mathbf{0.040}$ & $\mathbf{8.6}$ \\
\bottomrule
\end{tabular}
\end{table}

Table~\ref{tab:1d_vn} reports the results, together with an MCMC baseline run
on the same test functions under the same per-function sample budget. In this easy mixing regime, MCMC produces accurate samples efficiently,
and our model is less accurate though comparable in speed.  The
representative samples in Figure~\ref{fig:1d_vn_samples} show that it
nevertheless reproduces the target laws across the family, although it
systematically underestimates the variance of the invariant measure and the
Kolmogorov--Smirnov test rejects for most test functions.  This gap is
expected: an
amortized model that generalizes across the entire function family in a
single forward pass trades per-distribution accuracy for breadth.  

\begin{figure}[htbp]
\centering
\includegraphics[width=\textwidth]{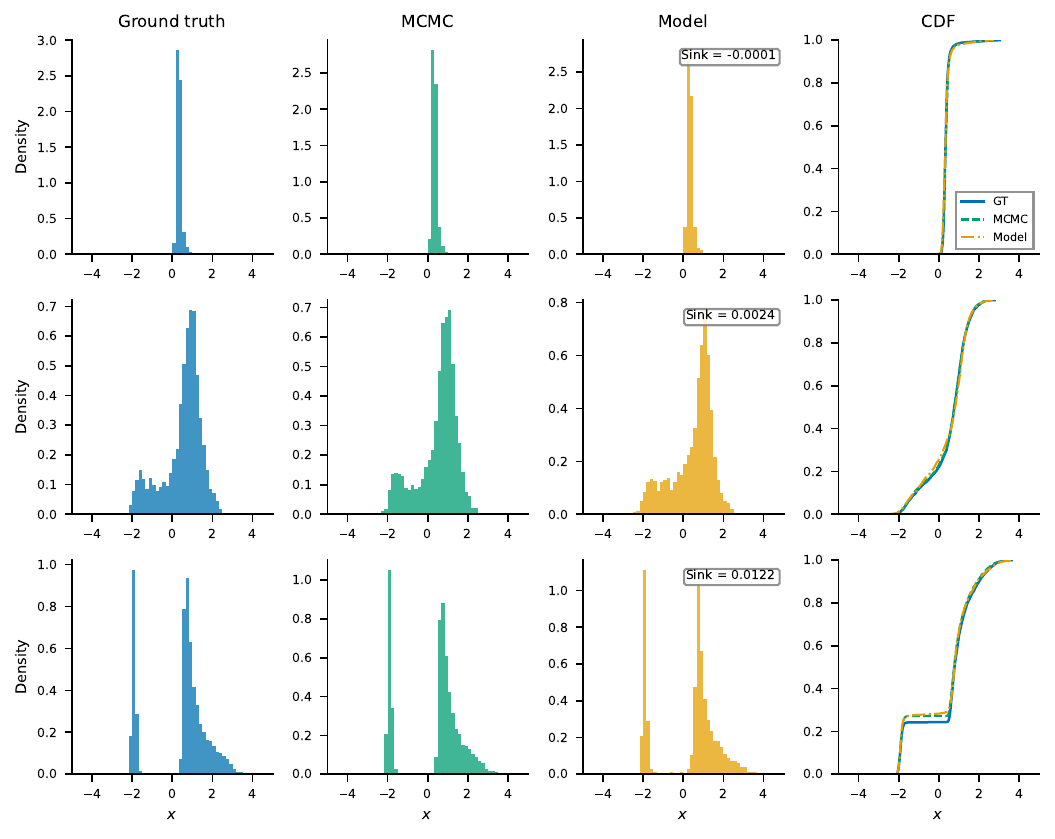}
\caption{\expVariableNoise. Representative samples on test functions at the
    25th, 50th, and 75th percentile Sinkhorn divergence. Each row shows one test
    function: reference histogram, MCMC histogram, model histogram,
    and CDF overlay.\label{fig:1d_vn_samples}}
\end{figure}

\subsubsection{\expRareEvent}\label{sec:rareevent}

\expRareEvent{} targets the regime where amortized operator sampling provides
its primary advantage: SDEs with small diffusion coefficients whose invariant
measures are multimodal. MCMC methods must traverse high free energy barriers
between modes, leading to mixing times that scale exponentially in~$1/\sigma^2$.
Our neural sampler bypasses this bottleneck entirely. Its inference cost is
independent of the mixing time.

We consider 1D SDEs whose invariant measure is a prescribed mixture of three
Gaussians:
\begin{equation}
p(x) = \sum_{i=1}^3 w_i\,\mathcal{N}(x\mid\mu_i,\sigma_i^2),
\end{equation}
where $w_i>0$, $\sum_i w_i=1$, $\mu_i\sim\mathrm{Unif}(-3,3)$, and
$\sigma_i\sim\mathrm{Unif}(0.1,0.5)$.  The diffusion coefficient is a small
random perturbation of the constant $0.1$, and the drift is constructed from
the zero flux condition of the stationary Fokker--Planck equation, so that
$p$ is by design the invariant density of the resulting SDE; the explicit
construction is given in Appendix~\ref{app:oned_details}.  Because the target
density is known, training samples are drawn directly from the GMM, making
the experiment a controlled test of the model's representational capacity and
generalization in a rare event regime.  The model reuses the architecture and
hyperparameters of \expVariableNoise.

Table~\ref{tab:1d_rare} compares the model with MCMC baselines at varying
time budgets, including a \emph{converged} baseline and three baselines that limit
the per-function MCMC budget.  The model attains a lower
mean Sinkhorn divergence than even the converged MCMC baseline while using
$7\times$ less elapsed time, and its median is comparable to the converged
MCMC median with far fewer large failures: the error distributions are
heavily skewed to the right (Figure~\ref{fig:1d_rare_box} in
Appendix~\ref{app:oned_details}).  The gap in means is driven by MCMC's
tendency to become trapped in a single mode when noise is small, occasionally
producing accurate samples (when all modes happen to be found) but often
yielding very poor ones; representative failure cases are shown in
Figure~\ref{fig:1d_rare_samples}.  Figure~\ref{fig:1d_rare_compute} shows the
Sinkhorn divergence as a function of elapsed time; our method (star)
dominates the MCMC Pareto frontier.

\begin{table}[htbp]
\centering
\caption{\expRareEvent. Sinkhorn divergence and $W_2$ on 1{,}024 test functions
  (mean\;/\;median).}
\label{tab:1d_rare}
\begin{tabular}{lccr}
\toprule
Method & Sinkhorn & $W_2$ & Time (s) \\
\midrule
\textbf{Ours} & $\mathbf{0.131}\;/\;0.011$ & $\mathbf{0.447}\;/\;0.239$ & $\mathbf{20.6}$ \\
MCMC ($6$\,s) & $0.348\;/\;0.039$ & $0.708\;/\;0.474$ & $6.3$ \\
MCMC ($12$\,s) & $0.326\;/\;0.019$ & $0.659\;/\;0.381$ & $12.4$ \\
MCMC ($25$\,s) & $0.299\;/\;0.009$ & $0.614\;/\;0.281$ & $24.6$ \\
MCMC (conv.) & $0.254\;/\;0.003$ & $0.535\;/\;0.161$ & $136.9$ \\
\bottomrule
\end{tabular}
\end{table}

\begin{figure}[htbp]
\centering
\includegraphics[width=\textwidth]{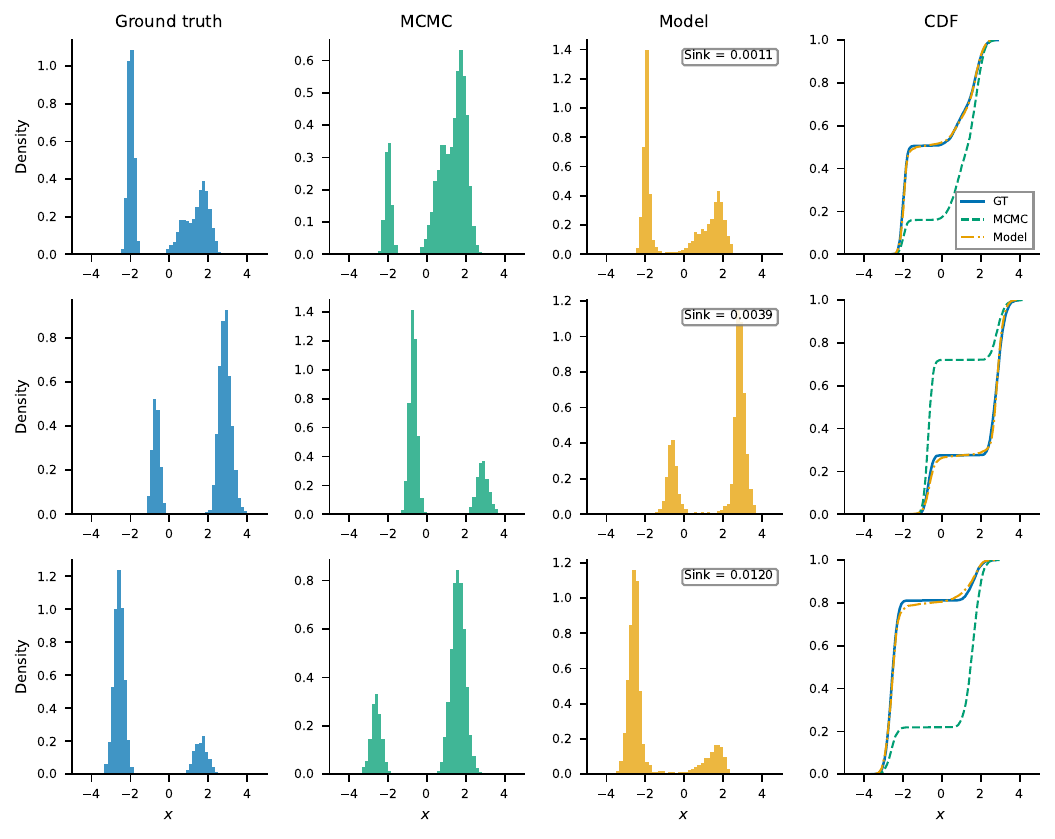}
\caption{\expRareEvent. Representative samples on test functions chosen to
    illustrate MCMC failures caused by mode trapping. Each row shows one test
    function: reference histogram, MCMC histogram, model histogram,
    and CDF overlay.\label{fig:1d_rare_samples}}
\end{figure}

\begin{figure}[htbp]
\centering
\includegraphics[width=0.8\textwidth]{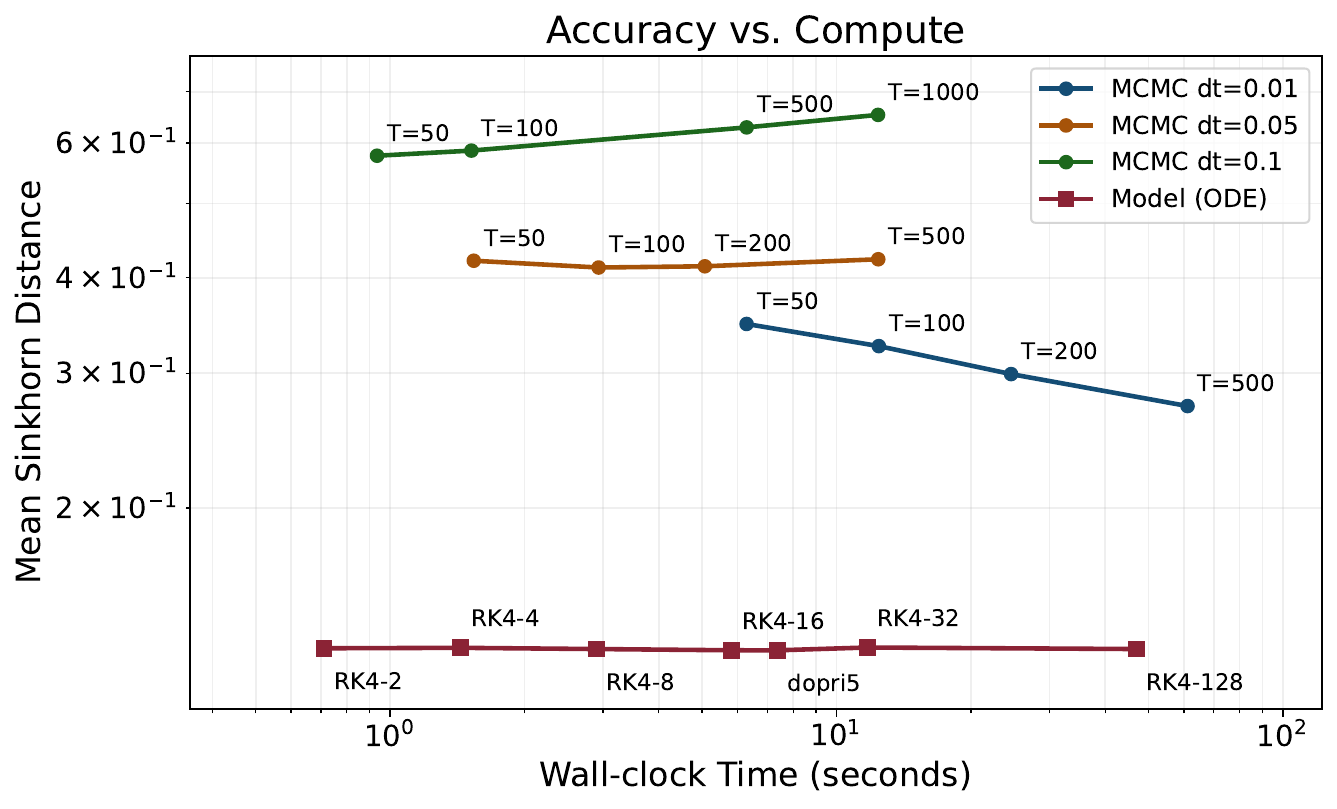}
\caption{\expRareEvent. Mean Sinkhorn divergence vs.\ elapsed time.
    Star: operator model. Circles: MCMC at increasing time
  budgets.\label{fig:1d_rare_compute}}
\end{figure}

\subsection{\expConstantNoise}\label{subsec:experiment}

This experiment compares grid DeepONet and Lagrangian sensor models for sampling invariant measures in 2D.

We consider amortized sampling from the invariant measures of the 2D SDE
\begin{equation}
dX = b(X)\,dt + \sqrt{2}\,dW, \qquad X\in\mathbb{R}^2,
\end{equation}
with constant isotropic diffusion. The drift is generated as
$b(x)=b_{\mathrm{GRF}}(x)+b_{\mathrm{well}}(x)$, where $b_{\mathrm{GRF}}$ is a Gaussian random field sampled on a $32\times 32$ grid over $[-5,5]^2$ (RBF kernel with length scale sampled from $\mathrm{Unif}(0.3,2.0)$), and $b_{\mathrm{well}}$ is a radial confining well that pushes particles back toward the origin outside radius~2 with strength~10.

Table~\ref{tab:d2_sinkhorn} compares three parts of the sampler.
The \emph{representation} determines which drift observations the encoder receives.
Grid DeepONet~\cite{jin_mionet_2022} reads the full drift grid.
Point probes retain the first observation of each trajectory.
Endpoint probes retain the first and last observations and their difference.
These two encoders use MLPs, while the trajectory encoder applies a CNN to all observations.

The \emph{aggregator} combines sensor embeddings into the context $z$.
MeanPool averages the embeddings before an MLP.
DeepSets~\cite{zaheer_deep_2017} applies a shared MLP to each embedding, sums the outputs, and applies another MLP.
Perceiver~\cite{jaegle_perceiver_2021} uses cross attention between learned latent tokens and sensor embeddings, as described in Section~\ref{subsec:perceiver_encoder}.

\emph{Conditioning} determines how the decoder uses the context.
FiLM (feature-wise linear modulation)~\cite{perez_film_2018} uses the context to predict a scale and shift for each hidden feature.
Our FiLM and AdaLN (adaptive layer normalization)~\cite{peebles_scalable_2023} variants both apply these changes after layer normalization.
Our AdaLN variant also uses the context to scale each residual update.
Additional variants and training details are in Appendix~\ref{app:twod_details}.

\begin{table}[t]
\centering
\small
\caption{\expConstantNoise. Mean Sinkhorn divergence on 100 test drift fields,
using 4{,}096 samples per distribution per field, after training on 65{,}536 fields.
Probe models use 256 sensors, with trajectory + Perceiver + AdaLN as the shared baseline.
Time: total training time under shared optimization settings.}
\label{tab:d2_sinkhorn}
\begin{tabular}{llcc}
\toprule
Component & Variant & Mean Sinkhorn & Train (s) \\
\midrule
Representation & Grid DeepONet & 0.0209 & 40.3 \\
& Point + Perceiver + AdaLN & 0.0415 & 87.8 \\
& Endpoint + Perceiver + AdaLN & 0.0324 & 82.8 \\
& Trajectory + Perceiver + AdaLN & 0.0180 & 252.0 \\
\midrule
Aggregator & Trajectory + MeanPool + AdaLN & 0.0370 & 226.2 \\
& Trajectory + DeepSets + AdaLN & 0.0246 & 238.4 \\
\midrule
Conditioning & Trajectory + Perceiver + FiLM & \textbf{0.0169} & 258.9 \\
\bottomrule
\end{tabular}
\end{table}

\begin{table}[t]
\centering
\small
\caption{\expConstantNoise. Sensor resolution transfer for the trajectory + Perceiver + AdaLN
backbone after training on 65{,}536 functions. Rows denote the
training sensor regime and columns denote the evaluation sensor count. The
mixed row alternates homogeneous batches from the $512$-, $1024$-, and
$2048$-sensor trajectory views. Each probe executes a 20-step random walk with
step size $0.05$. Lower Sinkhorn is better.}
\label{tab:d2_transfer}
\begin{tabular}{lccc}
\toprule
Train regime & Eval 512 & Eval 1024 & Eval 2048 \\
\midrule
Fixed 512 & 0.0284 & \textbf{0.0172} & 0.0131 \\
Fixed 1024 & 0.0287 & 0.0175 & \textbf{0.0125} \\
Fixed 2048 & 0.0312 & 0.0184 & 0.0131 \\
Mixed 512/1024/2048 & \textbf{0.0271} & 0.0177 & \textbf{0.0125} \\
\bottomrule
\end{tabular}
\end{table}

It turns out that the trajectory + Perceiver + AdaLN architecture is a practical choice with competitive accuracy. In the meanwhile, this architecure supports fixed or mixed counts in training and different counts at inference without retraining.

\subsection{\expInteractingParticle}\label{subsec:particles}

\expInteractingParticle{} is a proof of concept for sampling that depends on input in an
SDE in high dimensions that does not use a grid.  The central question is whether the learned
sampler uses observations of an unknown pair interaction law to generate
physically structured samples in a 64-dimensional particle system, beyond
matching the one body scale imposed by the confining potential.

We evaluate the framework on amortized sampling for $N=32$ interacting
particles in~2D ($X\in\mathbb{R}^{64}$) under overdamped Langevin dynamics
\begin{equation}\label{eq:langevin_d3}
dX = -\nabla V(X)\,dt + \sqrt{2\beta^{-1}}\,dW,
\end{equation}
targeting the Boltzmann distribution
$\pi(X)\propto\exp(-\beta V(X))$. The potential is
\begin{equation}\label{eq:potential_d4}
V(X) = \frac{\kappa}{2}\sum_{i=1}^{N}|x_i|^2 + \sum_{i<j}U_{\mathrm{int}}(|x_i-x_j|),
\end{equation}
with a fixed harmonic trap ($\kappa=0.5$) and inverse temperature
$\beta=2.0$ ($T=0.5$). Across problem instances, the interaction kernel
$U_{\mathrm{int}}$ varies; the one body term remains fixed.

We draw $U_{\mathrm{int}}$ from four qualitatively distinct kernel
families---Morse (bondlike attraction with a repulsive core), Gaussian (soft
repulsion or attraction), double Gaussian (short range repulsion with
long range attraction, creating a preferred interparticle spacing), and WCA
(hard sphere exclusion).  All kernels are nonsingular; their functional
forms and the energy capping convention are given in
Table~\ref{tab:d4_kernels} in Appendix~\ref{app:particle_details}.

Because $U_{\mathrm{int}}$ acts on pairwise distances, full system probes mix
many interactions at once.  The sensor design therefore uses \textbf{pair probe
sensors}: isolated trajectories of particle pairs that record the relative
displacement and interaction force along short rollouts.  The resulting
observation tensor is a fingerprint of the interaction kernel that does not
require an Eulerian grid in $\mathbb{R}^{64}$.  The model combines a 1D CNN
sensor encoder, a Perceiver aggregator with four latent tokens, and an
equivariant particle flow network conditioned by AdaLN; it has 1.78M
parameters and is trained by conditional flow matching on
2{,}048 kernels, with metrics reported on 100 test kernels (sensor,
data generation, and training details in Appendix~\ref{app:particle_details}).

The state $X=(x_1,\dots,x_{32})\in\mathbb{R}^{64}$ is invariant under
permutations of the indistinguishable particles.  We therefore evaluate the
generated distributions using physical summaries invariant under particle permutations:
pair correlation functions, energy distributions, center of mass and
radius of gyration statistics, pair distance summaries, and nearest neighbor
summaries.  Raw coordinate space Sinkhorn divergence is reported only as a
secondary diagnostic.  In 64 dimensions it can be misleading because the fixed
harmonic trap makes a pooled Gaussian baseline competitive in coordinate cost
even though it contains no interaction structure.

Table \ref{tab:d4_ablation} presents the quantitative evidence that our model has real prediction power over null. When the input is the correct sensors, the generated samples' distribution is far more closer to the ground truth then shuffled sensors input. And the generated sample also outperforms the pooled Gaussian (fitted to the target coordinate mean and variance), and standard Gaussian samples.

\begin{table}[ht]
\centering
\caption{\expInteractingParticle. Conditioning diagnostic on 100 test
interaction kernels. The score is sliced 2-Wasserstein on standardized physical
features invariant under permutations. Lower is better.}
\label{tab:d4_ablation}
\begin{adjustbox}{max width=\textwidth}
\begin{tabular}{lccc}
\toprule
Method & Mean & Median & Comparison \\
\midrule
Correct sensors & 1.7860 & 1.0159 & Model conditioned on the true kernel \\
Shuffled sensors & 30.1422 & 15.9140 & Correct wins $99/100$ \\
Pooled Gaussian & 23.1144 & 13.9604 & Correct wins $99/100$ \\
Random normal & 20.0154 & 12.3747 & Correct wins $98/100$ \\
\bottomrule
\end{tabular}
\end{adjustbox}
\end{table}

We can also observe that the generated samples reproduce the pair distance patterns within each family. The energy distribution, pooled across all test kernels, also has a similar shape to the target distribution. More quantitative results are collected in Table~\ref{tab:d4_obs_metrics} in
Appendix~\ref{app:particle_details}. Figure~\ref{fig:d4_gr} plots the
empirical pair correlation function $g(r)$ broken down by kernel family, and
Figure~\ref{fig:d4_energy} compares energy histograms.

\begin{figure}[ht]
    \centering
    \includegraphics[width=\textwidth]{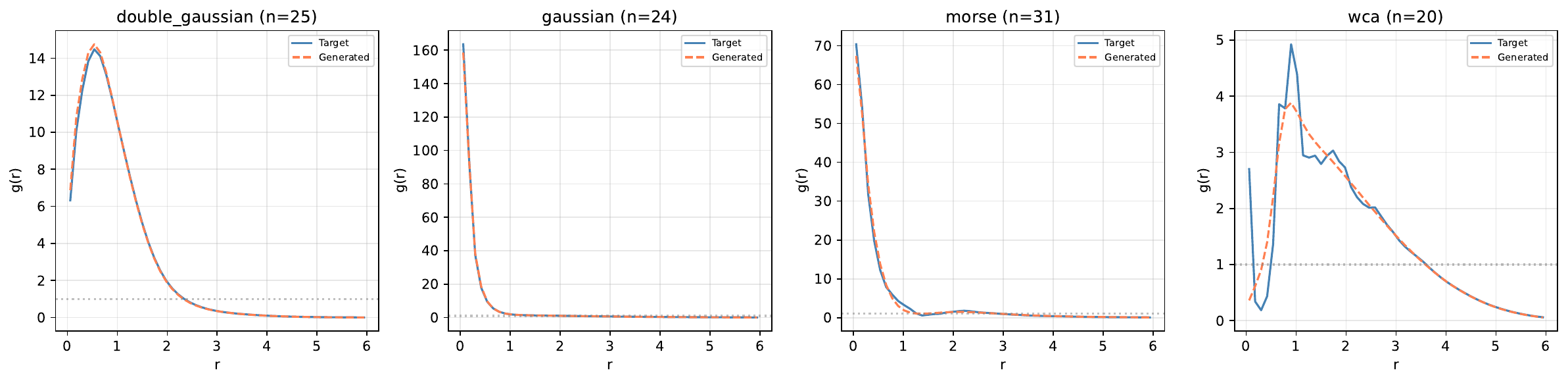}
    \caption{\expInteractingParticle. Pair correlation function $g(r)$ by kernel
    family: target (blue) vs.\ generated (dashed orange). Each subplot pools test
    kernels of one family. The curves are empirical histograms without additional smoothing.}
    \label{fig:d4_gr}
\end{figure}

\begin{figure}[ht]
    \centering
    \includegraphics[width=0.6\textwidth]{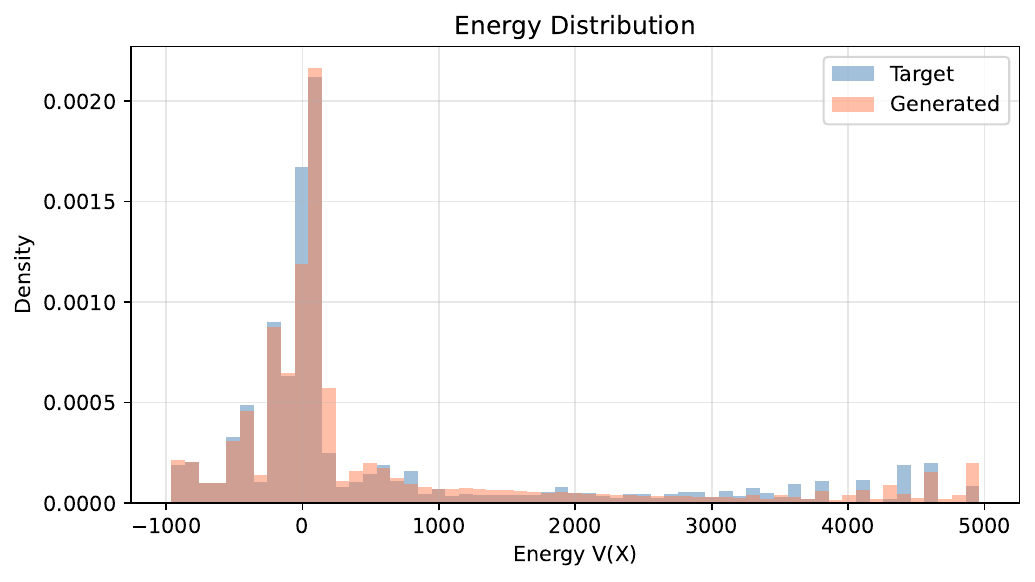}
    \caption{\expInteractingParticle. Energy histogram comparison: target
    (blue) vs.\ generated (orange).}
    \label{fig:d4_energy}
\end{figure}

Taken together, the random Lagrangian sensor architecture shows the potential to function as an amortized sampler of parametric families of high-dimensional SDEs.

\section{Conclusion}\label{sec:conclusion}

We present an amortized neural sampler that combines operator learning with continuous normalizing flows to sample from the invariant measures of parametric families of SDEs. The framework encodes SDE coefficients into a latent representation, then conditions a continuous normalizing flow that transports a Gaussian distribution to a target distribution. In practice, we adopt flow matching to train the model. Unlike the established DeepONet, our architecture ingests random Lagrangian trajectory sensors and uses an attention mechanism, which avoids grid discretization in spaces of high dimension. On the theory side, we establish a reduction theorem showing that the UAP of a neural operator implies the expressivity of a neural operator sampler. Thus, given a sufficiently powerful neural operator, our framework can fit the map from SDE coefficients to velocity fields that generate the invariant measures. We also prove that our framework is resolution invariant: the sampler remains consistent as the sensor count tends to infinity.

The \expVariableNoise{} and \expRareEvent{} experiments show the expected regime dependence: MCMC remains stronger in settings with easy, fast mixing; the amortized sampler becomes useful once slow mixing dominates the compute budget. In \expConstantNoise, the model with random Lagrangian sensors materially outperforms DeepONet. In \expInteractingParticle, the sampler demonstrates its capability in high dimensions.

Several directions remain open.  First, the supervised results rely on access to training samples from target invariant measures, which may require substantial offline simulation effort, and the amortized approach is most useful when many related SDE instances must be solved and mixing is slow. Data preparation, unsupervised training, and training without samples need further investigation. Second, the current framework targets the invariant measure, and a natural extension is to learn the full law $\mu_t$ that depends on time, or the \emph{weak solution} of the SDE.  Lastly, our experiments in higher dimensions are only a proof of concept. The network structure and training procedure have not been extensively optimized. Scaling to larger particle counts and configuration spaces of higher dimension remains an important challenge. 

\appendix

\section{Proofs of Main Results}\label{app:proofs}

\subsection{Stationary Fokker-Planck measures and stability}

\begin{proof}[Proof of Proposition~\ref{prop:stationary_fpk}]
Fix \(a=(b,\sigma)\in \mathcal{A}(\nu, p_0, \lambda, M) \), write \(A=A_a\), and let \(\mu_a\) be the corresponding stationary
measure. Existence and uniqueness follow by the standard result as commented in the text below the statement. Here, we only show how the moment bound follows by the Lyapunov condition.

For the moment bound, let \(V_p(x)=1+|x|^p\), \(2\le p\le p_0\). For
\(x\ne0\), the standard calculation gives
\[
\begin{aligned}
\mathcal L_a|x|^p &\le p|x|^{p-2}
    \left(\langle b(x),x\rangle+\frac12(p-1)\operatorname{Tr}A(x)\right) \\
    &\le
    p|x|^{p-2}
    \left(\langle b(x),x\rangle+\frac12(p_0-1)\operatorname{Tr}A(x)\right) \\
    &\le
    p|x|^{p-2}\bigl(-\lambda |x|^2+M\bigr)
    \le -c_p|x|^p+C_p ,
\end{aligned}
\]
where \(c_p,C_p>0\) are independent of \(a\). The nonsmooth point at \(x=0\)
is handled either by smoothing \(V_p\) near the origin or by applying the
following cutoff argument to \(1+(|x|^2+\varepsilon)^{p/2}\) and then
letting \(\varepsilon\downarrow0\).

Let \(\theta_n\) be increasing, concave, and bounded, with
\[
    0\le\theta_n'\le1,\qquad \theta_n''\le0,
\]
and with \(\theta_n'(r)\uparrow1\), \(\theta_n(r)\uparrow r\) on compact
intervals. Then
\[
\begin{aligned}
    \mathcal L_a\theta_n(V_p)
    &=
    \theta_n'(V_p)\mathcal L_aV_p
    +\frac12\theta_n''(V_p)
    \langle A\nabla V_p,\nabla V_p\rangle_{\mathbb R^d} \\
    &\le \theta_n'(V_p)(-c_pV_p+C_p).
\end{aligned}
\]
Integrating against \(\mu_a\), using stationarity, and then applying monotone
convergence gives
\[
    c_p\int V_p\,d\mu_a\le C_p ,
\]
uniformly over \(a\in \mathcal{A}(\nu, p_0, \lambda, M)\), which proves the claim.
\end{proof}

\begin{proof}[Proof of Lemma \ref{lemma:inv_stability}]

Let \(a_n\in\mathcal{A}(\nu, p_0, \lambda, M)\) and \(a_n\to a\) locally uniformly. Write \(A_n=A_{a_n}\) and \(A=A_a\). On every compact set, \(b_n\to b\) and \(A_n\to A\) uniformly, as noted in Section~\ref{sec:theoryexpress}. The two conditions that define \(\mathcal{A}(\nu, p_0, \lambda, M)\) are closed conditions on the values \((b(x),A_a(x))\) at each fixed \(x\). Hence they pass to the limit, and \(a\in\mathcal{A}(\nu, p_0, \lambda, M)\). This proves closedness.

Write \(\mu_n=\mu_{a_n}\) and \(\mu=\mu_a\). 
By the Lyapunov condition, $\int |x|^p d\mu_n$ is uniformly bounded.
This gives tightness. Take an arbitrary subsequence of \(\{\mu_n\}\). By
Prokhorov's theorem it has a further subsequence, not relabeled, such
that \(\mu_n\Longrightarrow\nu\) weakly for some probability measure
\(\nu\).

Let \(\varphi\in C_c^\infty(\mathbb R^d)\). Since \(\mu_n\) is
invariant,
\[
\int \mathcal L_{a_n}\varphi\,d\mu_n=0.
\]
On the compact support of \(\varphi\), \(b_n\to b\) and \(A_n\to A\) uniformly. Therefore
\[
\mathcal L_{a_n}\varphi
= b_n\cdot\nabla\varphi
+\frac12\operatorname{Tr}(A_nD^2\varphi)
\longrightarrow
\mathcal L_a\varphi
\]
uniformly.   Taking the limit $n\to\infty$, one then has
\[
\int \mathcal L_a \varphi\,d\nu=0
\qquad
\forall \varphi\in C_c^\infty(\mathbb R^d).
\]
By the density argument, this can be extended to $C_b^2$ test functions.  Hence, \(\nu\) is a invariant probability measure for the limiting SDE. The invariant measure is unique,
so \(\nu=\mu\).
Since every subsequence has a further subsequence converging weakly to
\(\mu\), the full sequence satisfies \(\mu_n\Rightarrow\mu\).

By the uniform $p$-moments, one has the uniform integrability of second moments, or 
\[
\lim_{R\to\infty}\sup_n \int_{|x|\ge R}|x|^2\mu_n(dx)\le \lim_{R\to\infty} R^{-(p_0-2)}M_{p_0}=0.
\]
Then, this together with the weak convergence implies that $\int |x|^2 d\mu_n\to \int |x|^2 d\mu $.
Hence \(W_2(\mu_n,\mu)\to0\), proving
continuity of \(a\mapsto\mu_a\). 
\end{proof}

\phantomsection\label{app:canonical_target}
\subsection{Proofs for the canonical target velocity}

We use the notation
\(\mu_a^S,\rho_t^a,J_t^a,v^a,\chi_R,v^{a,R}\), and \(\mu_a^{S,R}\)
introduced in Section~\ref{sec:theoryexpress}.

\begin{proof}[Proof of Lemma~\ref{lem:smoothed_velocity}]
First record the uniform moment input. Since
\(X\sim\gamma_\alpha=\mathcal N(0,(1-\alpha)I_d)\).
Let $D=Y-X$.
One has
\[
|D|^p\le 2^{p-1}(|X|^p+|Y|^p).
\]
Hence, one has $\sup \mathbb{E} |D|^p<\infty$ by the uniform control of $\mathbb{E} |Y|^p$. By the conditional expectation formula and Jensen's inequality:
\[
\int |v^a(z,t)|^p\,d\rho_t^{a}(z)
= \mathbb E\left|\mathbb E[D\mid Z_t^{a}]\right|^p
\le \mathbb E|D|^p  \le C_{\alpha,S,p},
\]
where the last bound follows from the exponential moment bound for
\(D\).

It remains to justify the pointwise linear growth estimate. 
Since \(X\mid (Z_t=z,Y=y)\) is Gaussian with mean
\[
m_{t,y,z}
=\frac{(1-t)(1-\alpha)}
{(1-t)^2(1-\alpha)+\alpha}\,(z-ty),
\]
one finds by the boundedness $|Y|\le S$  that
\[
|v^a(z,t)| \le \mathbb E[|Y|+|X|\mid Z_t^a=z] \le 
\mathbb{E}_{Y\mid Z_t^a=z} \mathbb E[|X|+|Y|\mid Z_t^a=z,Y]\le  C_{\alpha,S}(1+|z|).
\]
Lastly, differentiating the Gaussian kernel in \(t\), one has
\[
\partial_t\phi_\alpha(z-((1-t)x+ty))+\nabla_z\cdot\left((y-x)\phi_\alpha(z-((1-t)x+ty))\right)=0.
\]
Integrating this identity against \(\gamma_\alpha(dx)\mu_a^S(dy)\),
justified by dominated convergence using the Gaussian kernel and the
moment bounds above, gives the continuity equation.
\end{proof}

\begin{proof}[Proof of Lemma~\ref{lem:compact_target_continuity}]

We have shown that if \(a_n\to a\), then \(\mu_{a_n}^S\to\mu_a^S\) in \(W_2\). Write
\(\nu_n=\mu_{a_n}^S\) and \(\nu=\mu_a^S\). These measures are all supported in
\(\overline B_S\). Note that $\rho,J$, $\nabla_z\rho, \nabla_zJ$ ar obtained by integrating $\phi_{\alpha}, (y-x)\phi_{\alpha}$ and their $z$-derivatives with respect to $\gamma_{\alpha}(dx)\mu_a^S(dy)$.
These kernels are continuous in \((x,y,z,t)\), while $\gamma_{\alpha}$ is a Gaussian and $\mu_a^S$ is finitely supported. Hence they are dominated by an integrable
function of \(x\), uniformly over \(y\in\overline B_S\) and
\((z,t)\in\overline B_{2R}\times[0,1]\). Weak convergence of \(\nu_n\) gives
pointwise convergence of the corresponding integrals. Uniform continuity of
the kernels on
\[
\{|x|\le M,\ |y|\le S,\ |z|\le2R,\ t\in[0,1]\}
\]
plus the uniformly small Gaussian \(x\)-tail outside \(\{|x|\le M\}\) upgrades
this to uniform convergence on \(\overline B_{2R}\times[0,1]\).

There is also a uniform positive lower bound for \(\rho^a\). Indeed, for
\(|z|\le2R\), \(y\in\overline B_S\), \(t\in[0,1]\), and \(|x|\le1\),
\[
|z-((1-t)x+ty)|\le 2R+1+S .
\]
Therefore
\[
\rho_t^{a}(z)
\ge \gamma_\alpha(B_1)\,
(2\pi\alpha)^{-d/2} \exp\left(-\frac{(2R+1+S)^2}{2\alpha}\right)>0,
\]
uniformly in \(\nu\), \(z\), and \(t\). The quotient formula
\[
\nabla_z v^a =\frac{\rho^a\,\nabla_z J^a-J^a\otimes\nabla_z\rho^a}{(\rho^a)^2}
\]
therefore gives uniform convergence of the velocities and their first
\(z\)-derivatives on the compact cylinder. 

Recall that $v^{a,R}(x,t)=\chi_R(x)v^a(x,t)$ for \(a\in \mathcal{A}(\nu,p_0,\lambda,M)\).
Multiplication by the fixed smooth
cutoff preserves continuity in \(a\). It also yields a global Lipschitz bound:
on \(\overline B_{2R}\) the derivative of
\(\chi_Rv^a\) is uniformly bounded, while outside
\(\overline B_{2R}\) the field is identically zero. Hence
\(L_{\alpha,S,R}<\infty\).
\end{proof}

\begin{proof}[Proof of Lemma~\ref{lem:truncation_consistency}]
Let \(X_t^a\) solve the untruncated ODE
\[
\dot X_t^a=v^{a}(X_t^a,t),\qquad X_0^a\sim\gamma,
\]
and let \(X_t^{a,R}\) solve the truncated ODE
\[
\dot X_t^{a,R}=v^{a,R}(X_t^{a,R},t),\qquad X_0^{a,R}=X_0^a.
\]
The untruncated terminal law is \(\rho_1^a\), and the truncated
terminal law is \(\mu_a^{S,R}\).

Define
\[
\tau_R^a=\inf\{t\in[0,1]:|X_t^a|>R\}.
\]
On \(\{\tau_R^a>1\}\),  \(X_1^{a,R}=X_1^a\).  If \(\tau_R^a\le1\), then either \(|X_0^a|>R/2\) or
\[
\int_0^1 |v^a(X_t^a,t)|\,dt>R/2.
\]
Since \(X_t^a\) has law \(\rho_t^{a}\), the velocity moment bound
gives
\[
\sup_{a\in \mathcal{A}(\nu,p_0,\lambda,M)}\mathbb E\int_0^1
|v^a(X_t^a,t)|^p\,dt\le C_{\alpha,S,p}.
\]
Markov's inequality and the moments of \(X_0^a\) therefore give
\[
\sup_{a\in \mathcal{A}(\nu,p_0,\lambda,M)}\mathbb P(\tau_R^a\le1)\le C_{\alpha,S,p} R^{-p}
\]
for every \(p>2\).

Moreover, the linear growth bound and Gronwall's inequality give uniform
\(p\)-moment bounds for \(X_1^{a,R}\) and \(X_1^a\).
Using the coupling with a common initial point,
\[
W_2^2(\rho_1^{a},\mu_a^{S,R})
\le
\mathbb E|X_1^{a,R}-X_1^a|^2
\le C \mathbb P(\tau_R^a\le1)^{1-2/p}
\le C R^{-(p-2)}.
\]
This gives the claim.
\end{proof}

\subsection{Reduction and neural approximation}\label{app:express}

\begin{proof}[Proof of Proposition~\ref{prop:flow_stability}]
Couple the two flows by the same initial point
\(X_0\sim\rho_0\):
\[
X_t=\Phi_t(X_0),
\qquad
\hat X_t=\hat\Phi_t(X_0).
\]
Let \(\Delta_t=X_t-\hat X_t\). Then, one has
\[
\frac{d}{dt}|\Delta_t|^2\le 2L|\Delta_t|^2+2|\Delta_t|\delta,
\]
implying that \(|X_1-\hat X_1|\le e^L\delta\) almost surely. The joint law of
\((X_1,\hat X_1)\) is a coupling of the two terminal distributions, so
\[
W_2((\Phi_1)_\#\rho_0,(\hat\Phi_1)_\#\rho_0)
\le \left(\mathbb E|X_1-\hat X_1|^2\right)^{1/2}
\le e^L\delta.
\]
\end{proof}

\begin{proof}[Proof of Proposition~\ref{thm:reduction}]
Fix \(a\in K\). Let \(X_t\) solve the target truncated ODE
\[
    \dot X_t=v^{a,R}(X_t,t),
    \qquad X_0\sim\gamma,
\]
and let \(\widehat X_t\) solve the exact learned ODE
\[
    \dot{\widehat X}_t=\hat v_\theta[a](\widehat X_t,t),
    \qquad \widehat X_0=X_0.
\]
Put \(\beta=e^{L_{\alpha,S,R}}\delta_{\theta,\alpha,S,R}\). By assumption,
\(\beta<R\). Define
\[
    \tau=\inf\{t\in[0,1]:|\widehat X_t|\ge2R\}.
\]
On \([0,\tau)\), the learned path remains in \(B_{2R}\), where the
compact domain approximation bound applies. The same Gr\"onwall estimate as in
Proposition~\ref{prop:flow_stability}, stopped at \(\tau\), gives
\[
    |X_t-\widehat X_t|\le \beta,
    \qquad t\le\tau .
\]
If \(\tau\le1\), then \(|\widehat X_\tau|=2R\), and hence
\(|X_\tau|\ge2R-\beta>R\). Therefore
\[
    \{\tau\le1\}   \subset \left\{\sup_{t\le1}|X_t|\ge R\right\}.
\]
The target field satisfies
\(|v^{a,R}(x,t)|\le C_{\alpha,S}(1+|x|)\), and the learned field
satisfies the assumed global envelope. Gr\"onwall's inequality and Gaussian
moments of \(X_0\) give the uniform path moment bounds
\[
    \sup_{a\in K}\mathbb E\sup_{t\le1}|X_t|^p
    +\sup_{a\in K}\mathbb E\sup_{t\le1}|\widehat X_t|^p
    \le C_{\alpha,S,p}.
\]
Hence \(\P(\tau\le1)\le C_{\alpha,S,p}R^{-p}\), and H\"older's
inequality gives
\[
\begin{aligned}
    \mathbb E\left[
    |X_1-\widehat X_1|^2\mathbf 1_{\{\tau\le1\}}  \right]
    &\le  \left(\mathbb E(|X_1|+|\widehat X_1|)^p\right)^{2/p}
    \mathbb P(\tau\le1)^{1-2/p} \\
    &\le C_{\alpha,S,p}R^{-(p-2)} .
\end{aligned}
\]
On \(\{\tau>1\}\), the stopped estimate gives
\(|X_1-\widehat X_1|\le\beta\). Thus
\[
    W_2(\hat\mu_{a,\theta},\mu_a^{S,R})
    \le e^{L_{\alpha,S,R}}\delta_{\theta,\alpha,S,R} +C_{\alpha,S,p}R^{-(p-2)/2}.
\]
The triangle inequality, the truncation definition, the smoothing coupling,
and the projection estimate give
\[
\begin{aligned}
    W_2(\hat\mu_{a,\theta},\mu_a)
    &\le  e^{L_{\alpha,S,R}}\delta_{\theta,\alpha,S,R} +C_{\alpha,S,p}R^{-(p-2)/2}  +W_2(\mu_a^{S,R},\rho_1^a) \\
    &\quad  +W_2(\rho_1^a,\mu_a^S) +W_2(\mu_a^S,\mu_a) \\
    &\le
    e^{L_{\alpha,S,R}}\delta_{\theta,\alpha,S,R} +C_{\alpha,S,p}R^{-(p-2)/2}
    +\varepsilon_{\mathrm{trun}}(\alpha,S,R) +\sqrt{\alpha d}   +M_{p_0}^{1/2}S^{-(p_0-2)/2}.
\end{aligned}
\]
Taking the supremum over \(a\in K\) proves the theorem.
\end{proof}

\begin{proof}[Proof of Theorem~\ref{cor:conditional_expressivity}]
Fix \(\eta>0\) and \(p>2\). Choose \(S<\infty\) so large that
\[
    M_{p_0}^{1/2}S^{-(p_0-2)/2}<\eta/5 .
\]
Choose \(\alpha\in(0,1)\) so that \(\sqrt{\alpha d}<\eta/5\). For this
\((\alpha,S)\), choose \(R>0\) so large that
\[
    \varepsilon_{\mathrm{trun}}(\alpha,S,R)<\eta/5,
    \qquad  C_{\alpha,S,p}R^{-(p-2)/2}<\eta/5 .
\]
For the fixed triple \((\alpha,S,R)\), Lemma~\ref{lem:compact_target_continuity} shows that \(a\mapsto v^{a,R}\) is
continuous on \(K\). The operator
\(\Psi(a)=v^{a,R}\) is continuous from \(K\) to
\(C(\overline B_{2R}\times[0,1];\mathbb R^d)\). It also satisfies
\(|\Psi(a)(x,t)|\le C_{\alpha,S}(1+|x|)\). Apply
Assumption~\ref{ass:operator_uat} with \(C_{\rm env}=C_{\alpha,S}\), and
choose \(\theta\) such that
\[
    \delta_{\theta,\alpha,S,R}
    <  \min\left\{
    Re^{-L_{\alpha,S,R}},  \frac{\eta}{5e^{L_{\alpha,S,R}}}
    \right\}.
\]
Proposition~\ref{thm:reduction} gives
\[
    \sup_{a\in K}W_2(\hat\mu_{a,\theta},\mu_a)<\eta .
\]
\end{proof}

\begin{proof}[Proof of Proposition \ref{prop:lagrangian_uat}]
Let \(\Psi:K\to C(Y;\mathbb R^d)\) be a continuous target operator as in
Assumption~\ref{ass:operator_uat}, with \(|\Psi(a)(x,t)|\le C_{\rm env}(1+|x|)\). Fix
\(\delta>0\).

By compactness of \(K\), choose \(\varepsilon_\Psi>0\) such that
\(d_{\mathcal X_a}(a,a')<\varepsilon_\Psi\) implies
\(\|\Psi(a)-\Psi(a')\|_{C(Y)}<\delta/3\). Apply (FS) with this
\(\varepsilon_\Psi\) and write \(\mathcal S=\mathcal S_{m,L}\); the image
\(\mathcal S(K)\) is compact. Cover it by finitely many balls
\(B(\mathcal S(a_i),r/2)\), \(i=1,\ldots,N\), and take a continuous partition
of unity \(\{\chi_i\}_{i=1}^N\) subordinate to the enlarged balls
\(B(\mathcal S(a_i),r)\). Define

\[
    \bar G(s)(y)=\sum_{i=1}^N\chi_i(s)\,\Psi(a_i)(y),
    \qquad s\in\mathcal S(K),\ y=(x,t)\in Y.
\]
If \(\chi_i(\mathcal S(a))>0\), then
\(d_{\rm sens}(\mathcal S(a),\mathcal S(a_i))<r\), so (FS) gives
\(d_{\mathcal X_a}(a,a_i)<\varepsilon_\Psi\); hence

\[
    \sup_{a\in K}\|\bar G(\mathcal S(a))-\Psi(a)\|_{C(Y)}<\delta/3 .
\]
The map \((s,y)\mapsto\bar G(s)(y)\) is continuous on
\(\mathcal S(K)\times Y\) and invariant under permutations of the sensor argument.
By (EU), the finite sensor Perceiver--AdaLN network approximates it within
\(\delta/3\), uniformly over \(a\in K\) and \(y\in Y\). Combining the two
estimates gives

\[
    \sup_{a\in K}\sup_{(x,t)\in Y}
    |\mathcal N_\theta(a)(x,t)-\Psi(a)(x,t)|<\delta,
\]
the compact domain approximation part of Assumption~\ref{ass:operator_uat}.
The final statement follows by adding the global linear growth clause; local
Lipschitzness is already part of (EU).
\end{proof}

\begin{proof}[Proof of Theorem \ref{thm:randomFS}]

First we note that for general $L\ge 1$, one clearly has
\[
d_{\rm sens}(\mathcal S_{m,L}(a),\mathcal S_{m,L}(a'))
\ge d_{\rm sens}(\mathcal S_{m,1}(a),\mathcal S_{m,1}(a')).
\]
where $\mathcal S_{m,1}(a)$ means that we only keep the first point in each trajectory and the 
corresponding values of $a$. Then,
\[
\mathcal S_{m,1}(a)=\{(x_i, b(x_i), \rm vec~ \sigma(x_i))\}_{i=1}^m
\]
where $(x_i)$'s are sampled i.i.d from $\rho_0(x)\,dx$. Hence,
\[
d_{\rm sens}(\mathcal S_{m,1}(a),\mathcal S_{m,1}(a'))
=\max_{1\le i\le m}\|a(x_i)-a'(x_i)\|.
\]
Hence, it suffices to show with probability $1-\delta$ such that
\[
\max_{1\le i\le m}\|a(x_i)-a'(x_i)\|\le r \quad \Longrightarrow \quad 
d_{\mathcal{X}_a}<\varepsilon.
\]

We first take a radius $J$ so large that the tail of the metric
\(d_{\mathcal X_a}\) is smaller than the prescribed tolerance. On the compact
ball \(B_J\), the family \(K\) is uniformly equicontinuous. Therefore a
sufficiently fine spacing \(h\) gives that, if every point in \(B_J\) has some
sample point \(x_i\) with distance at most \(h\), then
\[
    d_{\mathcal X_a}(a,a')<\varepsilon .
\]
Hence, it suffices to show with probability $1-\delta$ that
$\{x_i\}_{i=1}^m$ forms a $h$-net of $B_J$.

We take $N$ balls with radius $h/2$ that cover $B_J$. The number of such balls can be chosen as
$N\le (CJ/h)^d$. If each ball contains at least one $x_i$, then $\{x_i\}_{i=1}^m$ forms a $h$-net of $B_J$.
For each ball $B_{\ell}$, the probability that it is empty is bounded by 
\[
(1-q_\ell)^m\le \exp(-m q_{\ell}),
\]
where 
\[
q_{\ell}=\int_{B_{\ell}}\rho_0\,dx.
\]
Taking union bound, the probability that there is one empty ball is controlled by
\[
p\le N \sup_{\ell}\exp(-N q_{\ell})\le (CJ/h)^d\exp(-m \min_{\ell}q_{\ell}).
\]
For $p\le \delta$, one needs
\[
m\gtrsim d (\min_{\ell}q_{\ell})^{-1}.
\]
This shows the first claim. 

The remaining claims are straightforward.
\end{proof}

\subsection{Resolution invariance proofs}\label{app:randomresolution}

We use the notation of Section~\ref{sec:resoinvar} for the attention maps
\(A^h_\ell\), \(H_\ell\), the self attention--MLP update \(G_\ell\), and the
encoder \(\operatorname{Encoder}\).
The finite and infinite sensor encodings are \(z_{a,m}\) and \(z_a\).

\begin{proof}[Proof of Lemma~\ref{lem:perceiver_consequences}]
Write \(\mathcal K=K_q\times K_e\), compact. Each \(\omega^h_\ell,F^h_\ell\) is
\(C^\infty\) on \(\mathcal K\), so
\[
    \omega_{\min}=\min_{\ell,h,\,\mathcal K}\omega^h_\ell,\quad
    \omega_{\max}=\max_{\ell,h,\,\mathcal K}\omega^h_\ell,\quad
    M_F=\max_{\ell,h,\,\mathcal K}\|F^h_\ell\|,\quad
    L_\omega=\max_{\ell,h,\,\mathcal K}\bigl(\|\nabla_q\omega^h_\ell\|+\|\nabla_q F^h_\ell\|\bigr)
\]
are finite, with \(\omega_{\min}>0\) since
\(\omega^h_\ell(q,e)=\exp(\langle q,\tilde eW^K_{\ell,h}\rangle/\sqrt{d_h})>0\); this gives
(i) and (ii), and integrating the lower bound gives
\(\int\omega^h_\ell(q,\cdot)\,d\nu\ge\omega_{\min}\) for every probability \(\nu\) on
\(K_e\).

For (iii), the \(\varepsilon\)-regularized LayerNorm normalizes each row to
Euclidean norm \(\le\sqrt{d_z}\) before its fixed affine map, so its output rows are
bounded, \(\|\mathrm{LN}(u)_r\|\le\rho_{\mathrm{LN}}<\infty\), and it is Lipschitz,
\(\operatorname{Lip}(\mathrm{LN})\le L_{\mathrm{LN}}<\infty\) (the variance is
\(\ge\varepsilon\)). Each \(G_\ell\) and each message addition step ends in a
LayerNorm, so with \(Z^{(0)}\) fixed the recursion \eqref{eq:Zrecursion} keeps every latent row at
norm \(\le\rho:=\max\{\rho_{\mathrm{LN}},\ \max_r\|Z^{(0)}_r\|\}\);
hence
\[
    K_z=\{Z\in\mathbb R^{k\times d_z}:\ \max_r\|Z_r\|\le\rho\},
    \qquad
    K_q=\{zW^Q_{\ell,h}:z\in K_z,\ \ell,h\}
\]
are compact and contain all reachable latents and queries. The query map is linear,
\(L^Q_\ell=\max_h\|W^Q_{\ell,h}\|\), and
\(G_\ell=\mathrm r[\mathrm{MLP}]\circ\mathrm r[\mathrm{SA}]\) composes
\(\varepsilon\)-LayerNorms, self attention, and the MLP, Lipschitz on \(K_z\) with
\(L^G_\ell=\operatorname{Lip}(G_\ell|_{K_z})\). For (iv), \(\operatorname{Decoder}\) is
\(C^1\), so
\[
    L^{\rm dec}_{\theta,R}=\sup_{K_z\times\overline B_R\times[0,1]}\|\nabla_z\operatorname{Decoder}\|<\infty .
\]
\end{proof}

\begin{proof}[Proof of Lemma~\ref{lem:attention_ulln}]
By Lemma~\ref{lem:perceiver_consequences}(ii), \(q\mapsto\omega^h_\ell(q,\cdot)\)
and \(q\mapsto[F^h_\ell(q,\cdot)]_r\) are Lipschitz from \(K_q\) into
\((C(K_e),\|\cdot\|_\infty)\); as \(K_q\) is compact and there are finitely many
\((\ell,h,r)\), \(\mathfrak F_\theta\) is a finite union of Lipschitz images of
compact sets, hence totally bounded in \(\|\cdot\|_\infty\), with envelope
\(M_\theta<\infty\). 

Fix \(\varepsilon>0\) and a finite \(\varepsilon\)-net
\(\{f_1,\ldots,f_N\}\subset\mathfrak F_\theta\). For every \(a\in K\),
\[
    \sup_{f\in\mathfrak F_\theta}
    \left| \int f\,d\hat\nu_{a,m}   -   \int f\,d\nu_a
    \right| \le  2\varepsilon  +  \max_{1\le j\le N}
    \left| \int f_j\,d\hat\nu_{a,m} -   \int f_j\,d\nu_a \right|.
\]
Since \(\hat\nu_{a,m}\) is the empirical law of \(m\) independent samples from
\(\nu_a\) and \(\|f_j\|_\infty\le M_\theta\), each term satisfies
\[
\sup_{a}\mathbb E_{\mathbf e}|\int f_j\,d\hat\nu_{a,m}-\int f_j\,d\nu_a|\le
M_\theta m^{-1/2},
\]
 so the maximum over the finite net is bounded by
\(NM_\theta m^{-1/2}\).

Hence \(\limsup_{m\to\infty}\delta_m^{\rm att}\le2\varepsilon\), and as \(\varepsilon>0\) is arbitrary, \(\delta_m^{\rm att}\to0\). 
\end{proof}

\begin{proof}[Proof of Theorem~\ref{prop:random_sensor_di}]

Note that for the resolution invariance, the parameters and architecture are fixed. 
We are increasing the number of probes (i.e., $m$ ) so that only the measure $\nu$ involved are varying.

\paragraph{Step 1: estimate for the encoder}

Let
\[
    \Delta_\ell   := \sup_{a\in K}  \mathbb E_{\mathbf e}  \|Z_\ell^{a,m}-Z_\ell^{a}\|,
\]
with \(\Delta_0=0\). We now derive the induction formula for $\Delta_{\ell}$.

For any \(\ell, h\), we recall that
\[
A^h_\ell(q,\nu)=\frac{\int F^h_\ell(q,\cdot)\,d\nu }{\int\omega^h_\ell(q,\cdot)\,d\nu}.
\]
By Lemma~\ref{lem:perceiver_consequences},
\[
\int\omega^h_\ell(q,\cdot)\,d\nu \ge \omega_{\min},
\quad \|\int F^h_\ell(q,\cdot)\,d\nu \|\le M_F,
\]
where the bounds are uniform in $\ell, h$.
Introducing \(\mathfrak F_{\ell,h}=\{\omega^h_\ell(q,\cdot):q\in K_q\}\cup\{[F^h_\ell(q,\cdot)]_r:q\in K_q,\ r=1,\dots,d_h\}\),
one then has
\[
\|N^\nu_q-N^{\tilde\nu}_q\|\le\sqrt{d_h}\,\sup_{f\in\mathfrak F_{\ell,h}}|\int f\,d\nu-\int f\,d\tilde\nu|,
\quad |D^\nu_q-D^{\tilde\nu}_q|\le \sup_{f\in\mathfrak F_{\ell,h}}|\int f\,d\nu-\int f\,d\tilde\nu|.
\]
Hence, the change of measure $\nu$ will bring in error as follows
\[
    \|A^h_\ell(q,\nu)-A^h_\ell(q,\tilde\nu)\|
    \le  \Bigl(\frac{\sqrt{d_h}}{\omega_{\min}}+\frac{M_F}{\omega_{\min}^2}\Bigr)
     \sup_{f\in\mathfrak F_{\ell,h}}|\int f\,d\nu-\int f\,d\tilde\nu|.
\]
Together  with Lemma~\ref{lem:perceiver_consequences} (ii),  one has the total error brought by the change of $q$
and change of measure:
\begin{equation}\label{app:stabilityaux1}
\|A^h_\ell(q,\nu)-A^h_\ell(q',\tilde\nu)\|\le L_A\|q-q'\|+C_A\, \sup_{f\in\mathfrak F_{\ell,h}}|\int f\,d\nu-\int f\,d\tilde\nu|,
\end{equation}
with \(C_A=\sqrt{d_h}/\omega_{\min}+M_F/\omega_{\min}^2\) and
\(L_A=L_\omega/\omega_{\min}+M_F L_\omega/\omega_{\min}^2\).

Since
\[
 H_\ell(Z,\nu) =\bigl[\,A^h_\ell(z W^Q_{\ell,h},\nu)\,\bigr]_{h=1}^{H}\,W^O_\ell,
 \]
one has by \eqref{app:stabilityaux1}  and the
mean uniform law of large numbers in Lemma~\ref{lem:attention_ulln} that
\[
    \sup_{a\in K} \mathbb E_{\mathbf e}  \|H_\ell(Z_\ell^{a,m},\hat\nu_{a,m})  - H_\ell(Z_\ell^{a},\nu_a)\|
    \le  C_\ell\Delta_\ell   + C_\ell\delta_m^{\rm att}.
\]
Recall \eqref{eq:Zrecursion}, using the Lipschitz bounds for \(G_\ell\) one has,
\[
    \Delta_{\ell+1}  \le C_\ell\Delta_\ell + C_\ell\delta_m^{\rm att}.
\]
The depth \(L_{\rm enc}\) is fixed, so one easily derive by induction that
\[
    \sup_{a\in K} \mathbb E_{\mathbf e}
    \|z_{a,m}-z_a\|   \le  C_\theta\delta_m^{\rm att}.
\]

\paragraph{Step 2: estimate for the encoder and the terminal law}

Fix a localization radius \(R\) to be determined. The decoder Lipschitz condition (Lemma \ref{lem:perceiver_consequences}) on
compact \(x\)-cylinders gives
\[
\begin{aligned}
    \eta_m(R)  &= \sup_{a\in K} \mathbb E_{\mathbf e}
    \sup_{(x,t)\in\overline B_R\times[0,1]}  |v_{a,m}(x,t)-v_a(x,t)| \\
    &\le  L^{\rm dec}_{\theta,R} \sup_{a\in K}
    \mathbb E_{\mathbf e}\|z_{a,m}-z_a\|  \le  C_{\theta,R}\delta_m^{\rm att}.
\end{aligned}
\]

Since the finite and infinite sensor flows start from the same \(X_0\sim\gamma\), one has for fixed \(a\) and \(\mathbf e\) that
\[
W_2^2(\hat\mu_{a,m},\hat\mu_a)\le\mathbb E_{X_0}|X_1^{a,m}-X_1^{a}|^2.
\]
 Let
\(\tau_R\) be the first time either path exits \(\overline B_R\). On
\(\{\tau_R\ge1\}\) both paths stay in \(\overline B_R\), so Gr\"onwall and the local
\(x\)-Lipschitz bound in \textnormal{(A2)} give the pathwise estimate
\[
    |X_1^{a,m}-X_1^{a}|\,\mathbf 1_{\{\tau_R\ge1\}}
    \le e^{L^x_{\theta,R}}  \sup_{(x,t)\in\overline B_R\times[0,1]}  |v_{a,m}(x,t)-v_a(x,t)| .
\]
The linear growth in \textnormal{(A2)} easily gives, for any fixed \(p>2\),
\[
\mathbb E\sup_{t\le1}|X_t|^p\le C_{\theta,p}
\]
uniformly in \(a\) and \(\mathbf e\) and for both $\hat{\mu}_{a,m}$ and $\hat{\mu}_a$. Hence,
for any \(a\) and \(\mathbf e\),
\(\mathbb P(\tau_R<1)\le C_{\theta,p}R^{-p}\), and thus
\[
    \mathbb E\bigl[\,|X_1^{a,m}-X_1^{a}|^2;\ \tau_R<1\,\bigr]
    \le  \bigl(\mathbb E|X_1^{a,m}-X_1^{a}|^p\bigr)^{2/p}\,
    \mathbb P(\tau_R<1)^{1-2/p}
    \le C_{\theta,p}R^{-(p-2)} .
\]
Combining the results above, one has
\[
    \sup_{a\in K}
    \mathbb E_{\mathbf e}\,   W_2(\hat\mu_{a,m},\hat\mu_a)
    \le
    e^{L^x_{\theta,R}}C_{\theta,R}\delta_m^{\rm att}  + C_{\theta,p}R^{-(p-2)/2}.
\]

Fixing \(R\) and
letting \(m\to\infty\), Lemma~\ref{lem:attention_ulln} gives \(\delta_m^{\rm att}\to0\);
letting then \(R\to\infty\) proves random sensor resolution invariance in
expectation.
\end{proof}

\section{Experimental Details}\label{app:expdetails}

We compute Sinkhorn divergence with GeomLoss: $p=2$ selects quadratic cost, and \texttt{blur}${}=0.05$ sets $\varepsilon=\texttt{blur}^2$.

All experiments are run on a single NVIDIA RTX~5090 GPU (32\,GB) with an AMD
EPYC~9575F CPU (128~cores), using Python~3.14, PyTorch~2.10, and CUDA~12.8.
Elapsed times include all overhead (data loading, compilation, I/O).
The RTX~5090 and the EPYC~9575F have comparable list prices
(\$\!2{,}000 and \$\!2{,}200 respectively).  The runtime comparisons between
GPU neural sampling and CPU MCMC use hardware of equivalent cost.

\subsection{1D experiments}\label{app:oned_details}

The two 1D experiments share the Multi-input DeepONet architecture of Section~\ref{sec:perceiver}: latent dimension $q=128$, base channels~32 in the branch CNN, a trunk MLP of depth~4 and width~256 with GELU activations, and a sinusoidal time embedding of dimension~32. Training uses conditional flow matching with minibatch optimal transport coupling of the source--target pairs~\cite{lipman_flow_2024}. Table~\ref{tab:1d_params} summarizes the data generation and optimization settings of the two experiments side by side.

\begin{table}[htbp]
\centering
\caption{Configuration of the two 1D experiments.}
\label{tab:1d_params}
\begin{adjustbox}{max width=\textwidth}
\begin{tabular}{lcc}
\toprule
 & \expVariableNoise & \expRareEvent \\
\midrule
Function instances & 65{,}536 & 65{,}536 \\
Train\,/\,test split & 64{,}512\,/\,1{,}024 & 64{,}512\,/\,1{,}024 \\
Grid resolution over $[-5,5]$ & 64 & 256 \\
Reference samples per function & 4{,}096 & 4{,}096 \\
Reference generation & Euler--Maruyama & direct GMM draws \\
\quad time step $\Delta t$ & 0.01 & --- \\
\quad equilibration time $T_{\mathrm{burn}}$ & 50 & --- \\
\quad collection interval & 0.1 & --- \\
\midrule
Epochs (gradient steps) & 1 (32{,}256) & 1 (32{,}256) \\
Optimizer (lr, weight decay) & \multicolumn{2}{c}{AdamW ($10^{-3}$, $10^{-4}$), OneCycleLR} \\
Batch size\,/\,samples per function & \multicolumn{2}{c}{256\,/\,32} \\
Gradient clipping\,/\,precision & \multicolumn{2}{c}{1.0\,/\,bfloat16 mixed} \\
Training time & 263.1\,s & 296.5\,s \\
\bottomrule
\end{tabular}
\end{adjustbox}
\end{table}

In \expVariableNoise, the GRF perturbations $s,s'$ are sampled on the uniform grid via the FFT, and the reference Euler--Maruyama simulation runs on the linearly interpolated drift and diffusion. The MCMC baseline in Table~\ref{tab:1d_vn} runs Euler--Maruyama with the same $T_{\mathrm{burn}}=50$ on the 1{,}024 test functions, collecting 4{,}096 samples each. Figure~\ref{fig:1d_vn_training} shows the training loss curve.

In \expRareEvent, the diffusion coefficient is a small random perturbation of a constant,
\begin{equation}
\sigma(x)=0.1+0.1\sum_{k=1}^2 a_k\cos(b_k x+c_k),
\end{equation}
and the drift follows from the zero flux condition of the stationary Fokker--Planck equation
$\partial_x(b\,p - \tfrac{1}{2}\partial_x(\sigma^2 p))=0$, which yields
\begin{equation}
b(x) = \tfrac{1}{2}\sigma(x)^2\,\nabla\log p(x) + \sigma(x)\,\nabla\sigma(x).
\end{equation}
Because the target density $p(x)$ is a known Gaussian mixture, training samples are drawn directly from the GMM; this makes the experiment a controlled test of the model's representational capacity and generalization in a rare event regime, and complete sample generation from the SDE is left for the other experiments. The MCMC baselines in Table~\ref{tab:1d_rare} run Euler--Maruyama with $\Delta t=0.001$ on the linearly interpolated drift and diffusion for each of the 1{,}024 test functions, collecting 4{,}096 samples: the \emph{converged} baseline uses $5\times 10^5$ steps per function (elapsed time 136.9\,s), and three baselines with controlled runtimes limit the MCMC budget per function.

For inference we compared RK4 with 4 and 16 integration steps against the adaptive dopri5 solver. RK4 with only 4~steps yields Sinkhorn~$=0.131$ in $1.4$\,s (1{,}024 functions), matching dopri5 (Sinkhorn~$=0.130$, $7.4$\,s) at $5\times$ lower cost; RK4-4 is therefore the default solver in all 1D experiments. Figure~\ref{fig:1d_rare_box} shows the distribution of Sinkhorn divergence across test functions for \expRareEvent: the model's error distribution is concentrated near zero, while MCMC exhibits a heavy right tail caused by failures from mode collapse.

\begin{figure}[htbp]
\centering
\includegraphics[width=0.5\textwidth]{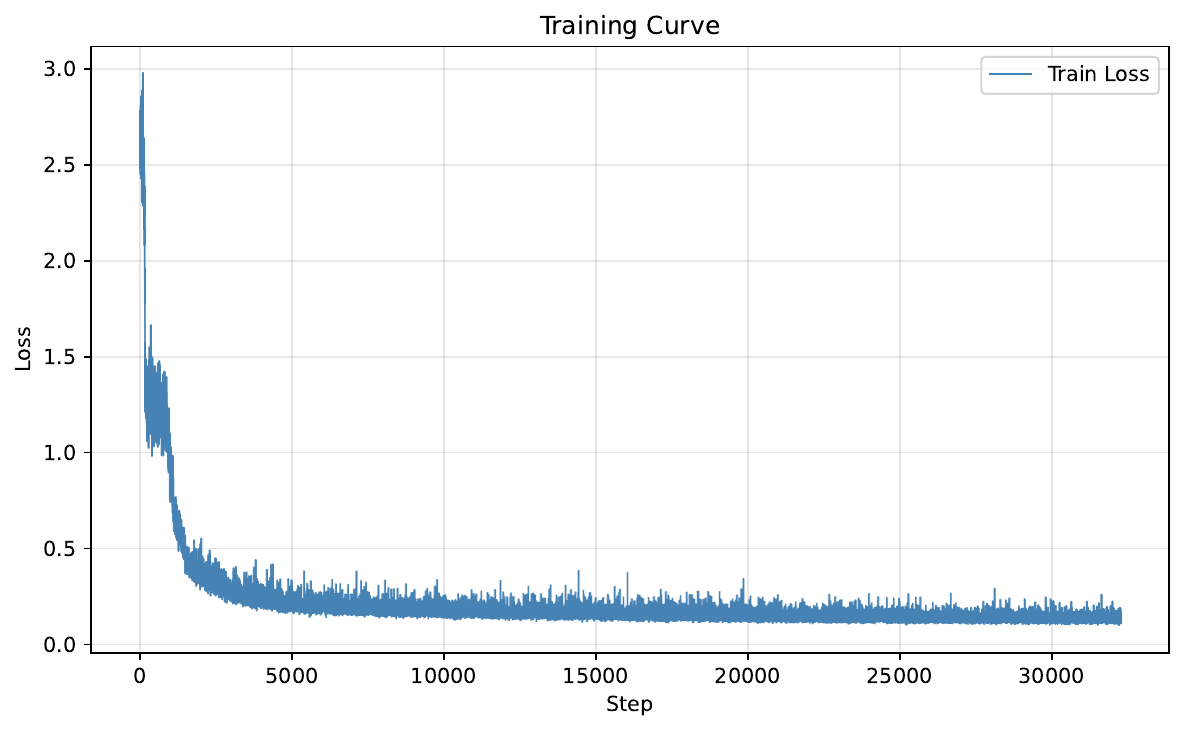}
\caption{\expVariableNoise. Training loss.\label{fig:1d_vn_training}}
\end{figure}

\begin{figure}[htbp]
\centering
\includegraphics[width=0.7\textwidth]{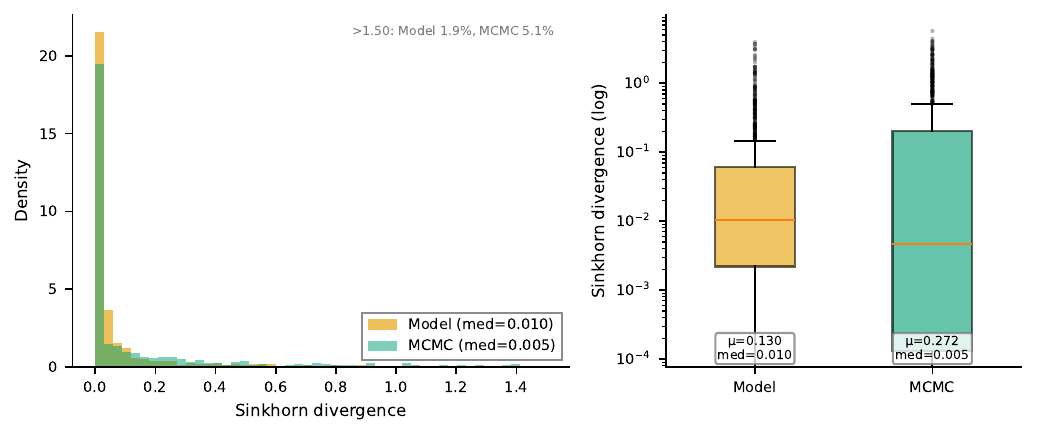}
\caption{\expRareEvent. Distribution of Sinkhorn divergence across 1{,}024 test functions:
    operator model vs.\ MCMC (converged). Left: histogram clipped at the 95th
    percentile; right: box plot on log scale.\label{fig:1d_rare_box}}
\end{figure}

\subsection{\expConstantNoise}\label{app:twod_details}

We generate 65{,}536 drift fields for training and 4{,}096 additional drift fields for evaluation. For each drift field, reference stationary samples are obtained by Euler--Maruyama with $dt=0.01$ and an equilibration period of 20 time units (2{,}000 steps), after which 4{,}096 samples are collected per function. All reported metrics are Sinkhorn divergences between 4{,}096 generated and 4{,}096 reference samples, computed on 100 test drift fields drawn from the evaluation split.

Probe data contain 20 steps with 8 features per step. Point and endpoint models use MLP encoders, and trajectory models use a 1D CNN. All use latent width 128 and a residual flow decoder trained with conditional flow matching and minibatch OT coupling. Each variant is trained for one epoch over the 65{,}536-function training split (8{,}192 optimizer steps at batch size 512) using OneCycleLR with peak learning rate $10^{-3}$, weight decay $10^{-4}$, gradient clipping 1.0, \texttt{torch.compile}, and bfloat16 mixed precision.

In Table~\ref{tab:d2_extra}, Concat appends the context to the hidden features at each decoder block.
Dot uses context-dependent weights to multiply the hidden features, following the branch--trunk product in DeepONet~\cite{lu_deeponet_2021}.
The sensor budget variants change the training count and retain 256 evaluation sensors.

\begin{table}[htbp]
\centering
\small
\caption{\expConstantNoise. Additional variants trained on 65{,}536 fields and evaluated on 100 test fields with 4{,}096 generated and reference samples each.
All use trajectory probes, Perceiver, and 256 evaluation sensors.}
\label{tab:d2_extra}
\begin{tabular}{lccc}
\toprule
Conditioning & Train sensors & Mean Sinkhorn & Train (s) \\
\midrule
Dot & 256 & 0.0203 & 248.6 \\
Concat & 256 & 0.0189 & 253.9 \\
AdaLN & 64 & 0.0180 & 180.1 \\
AdaLN & 128 & 0.0187 & 303.4 \\
\bottomrule
\end{tabular}
\end{table}

In the sensor resolution transfer study (Table~\ref{tab:d2_transfer}), each probe executes a 20-step random walk with step size $0.05$. For the $2048$-sensor and mixed count rows we keep the effective optimizer batch fixed at $512$ by using microbatch $256$ with gradient accumulation over two steps; the rows with high resolution fit on a single GPU without changing the optimization budget. Figure~\ref{fig:d2_samples} shows qualitative comparisons of target and generated samples for representative test drift fields.

\begin{figure}[htbp]
\centering
\includegraphics[width=\textwidth]{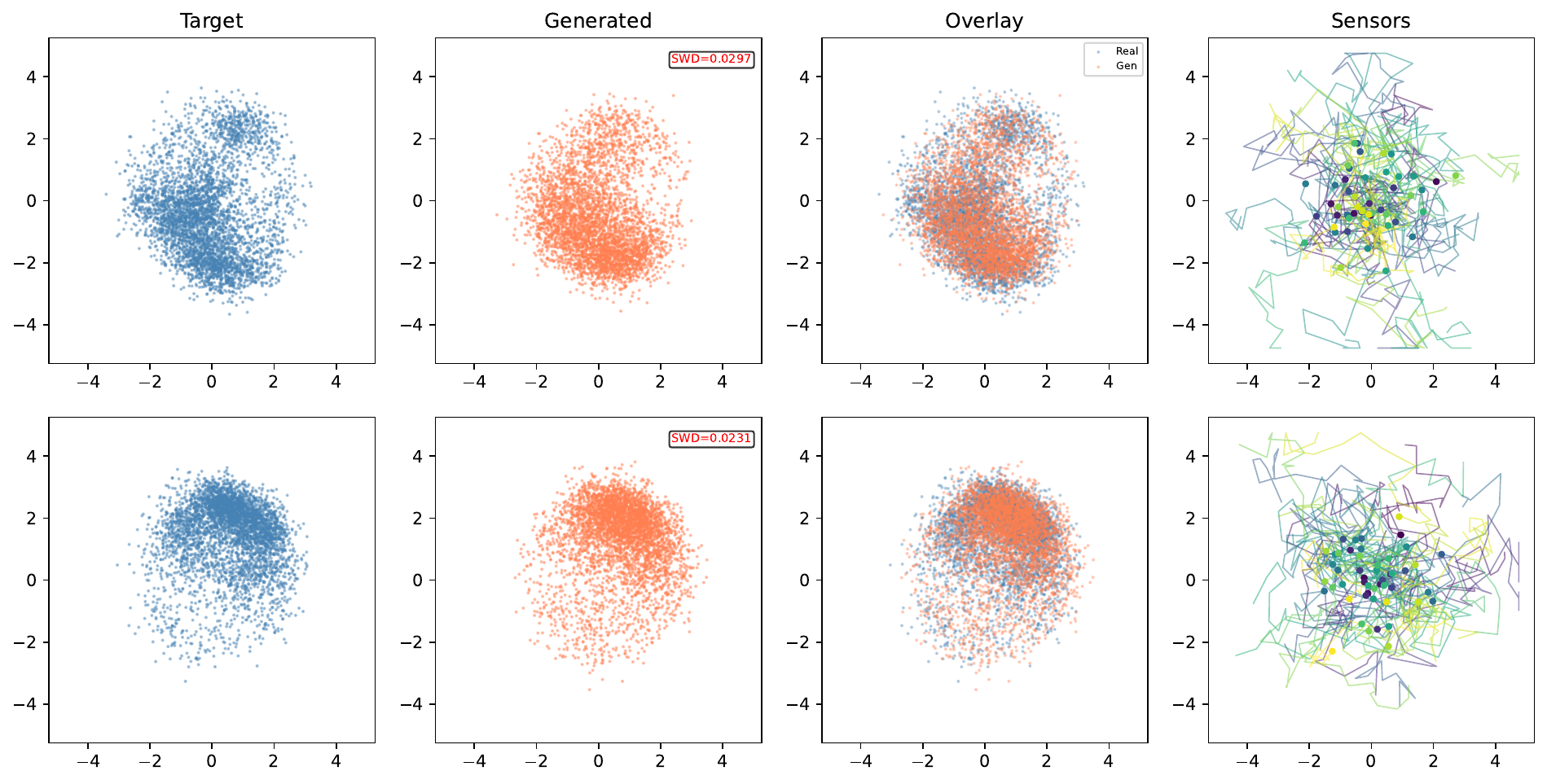}
\caption{\expConstantNoise. Qualitative comparison of target (ground truth) and
generated samples for representative test drift fields.}\label{fig:d2_samples}
\end{figure}

\subsection{\expInteractingParticle}\label{app:particle_details}

Table~\ref{tab:d4_kernels} lists the four interaction kernel families. All kernels are nonsingular. For Morse and WCA, we cap the energy at $U_{\mathrm{MAX}}=10$ to avoid numerical instabilities during MCMC equilibration; at $\beta=2$, the corresponding Boltzmann weight is already $\exp(-\beta U_{\mathrm{MAX}})\approx 2\times10^{-9}$.

\begin{table}[htbp]
\centering
\caption{Interaction kernel families for \expInteractingParticle.}
\label{tab:d4_kernels}
\begin{adjustbox}{max width=\textwidth}
\begin{tabular}{lll}
\toprule
Family & Form & Character \\
\midrule
Morse & $U(r) = D_e(1-e^{-a(r-r_e)})^2 - D_e$ & bondlike attraction with repulsive core \\
Gaussian & $U(r)=\alpha\exp(-r^2/2\sigma^2)$ & soft repulsion ($\alpha>0$) or attraction ($\alpha<0$) \\
Double Gaussian & short range repulsion $+$ long range attraction & preferred interparticle spacing \\
WCA & truncated and shifted Lennard--Jones with softening & hard sphere exclusion \\
\bottomrule
\end{tabular}
\end{adjustbox}
\end{table}

The sensor set for pair probes consists of 512 isolated trajectories of particle pairs per kernel. Each probe starts at a random separation $r_0\in[0.2,4.0]$, is evolved for 20 short steps ($dt=0.001$), and records the relative displacement and interaction force, giving a sensor tensor of shape $(512,20,4)$ per kernel. Reference samples are generated by Euler--Maruyama with $dt=0.001$, 16 chains per system, an equilibration period of $10^5$ steps, and 4{,}096 retained samples per kernel. The training set contains 2{,}048 kernels and the evaluation set contains 512 test kernels; we report metrics on 100 test kernels.

The architecture is the \texttt{ParticleAmortizedSampler} of Section~\ref{sec:perceiver}: a 1D CNN sensor encoder, a Perceiver aggregator with four latent tokens, and an equivariant particle flow network conditioned by AdaLN. The model has 1.78M parameters and is trained for one epoch (2{,}048 optimizer steps) by conditional flow matching with learning rate $10^{-3}$, so the product of learning rate and step count is order one. On a local RTX 5090 system, the full data generation, training, evaluation, and plotting pipeline completes in roughly 40 minutes; dataset generation dominates the cost.

As a scale reference for the conditioning diagnostic in Table~\ref{tab:d4_ablation}, comparing two independent halves of the MCMC target sample under the same sliced 2-Wasserstein score gives a noise floor from the target split of 0.2038 (mean) and 0.1071 (median); the correctly conditioned sampler (1.7860\,/\,1.0159) thus remains above the finite sample floor of the metric. We attribute this gap primarily to the behavior of sliced 2-Wasserstein on finite samples in 64 dimensions rather than to a correspondingly large sampling error; the pair distance distributions in Figure~\ref{fig:d4_gr} give the more direct evidence that the sampler conditions on the interaction kernel.

Table~\ref{tab:d4_obs_metrics} reports the observable checks invariant under permutations that are summarized in Section~\ref{subsec:particles}.

\begin{table}[htbp]
\centering
\caption{\expInteractingParticle. Observable checks invariant under
permutations, evaluated on 100 test kernels with 4{,}096 samples per kernel.}
\label{tab:d4_obs_metrics}
\begin{tabular}{lc}
\toprule
Metric & Value \\
\midrule
Pair correlation $g(r)$ $L^2$ error & 0.6865 \\
Mean marginal KL & $1.21\times 10^{-3}$ \\
Energy KS statistic & 0.1585 \\
Energy mean: model / target ($\Delta$, rel.) & $614.40$ / $638.76$ ($-24.36$, $-3.8\%$) \\
Energy std: model / target ($\Delta$, rel.) & $1472.60$ / $1504.73$ ($-32.13$, $-2.1\%$) \\
\bottomrule
\end{tabular}
\end{table}

\bibliographystyle{plain}
\bibliography{refs}

\end{document}